\documentclass[psamsfonts,reqno]{amsart}
\RequirePackage[paperwidth=168mm,paperheight=238mm,top=18mm,bottom=12mm,left=24mm,right=16mm,twoside]{geometry}
\usepackage{amssymb,amsfonts}
\usepackage{amsmath}
\usepackage{fourier}
\usepackage[all,arc]{xy}
\usepackage{enumerate}
\usepackage{mathrsfs}
\usepackage{graphicx}
\usepackage{wrapfig}
\usepackage[foot]{amsaddr}
\graphicspath{{./Figures/},{./Pictures/}}
\everymath{\displaystyle}
\usepackage{xcolor}
\usepackage{hyperref}

\usepackage{lipsum}
\hypersetup{
    colorlinks=true,
    linktoc=all,
    linkcolor=blue
}
\newtheorem{thm}{Theorem}[section]

\newtheorem{prop}{Proposition}
\newtheorem{lem}{Lemma}

\theoremstyle{definition}

\theoremstyle{remark}
\newtheorem{rem}[thm]{Remark}

\numberwithin{equation}{section}
\usepackage{lipsum} 
\usepackage{marvosym}
\usepackage{ifsym}
\begin{document}
\title[Stabilization of Piezoeletric Beam with  Coleman or Pipkin Gurtin Law]{Stabilization of Piezoeletric Beam  with Coleman or Pipkin Gurtin Thermal Law and under Lorenz Gauge Condition}

\author{Dounya Kechiche$^{*,\textrm{\Letter}}$ and Ammar Khemmoudj$^{*}$}

\address{D. Kechiche$^{*}$: Department of Mathematics, University of Science and Technology Houari Boumediene, Algiers, Algeria.}
\email{kechichdounia19@gmail.com$^{\textrm{\Letter}}$}
\email{akhemmoudj@yahoo.fr}
\newcommand{\Addresses}





\maketitle

\begin{abstract}
In this paper, we present the analysis of stability for a piezoelectric beam subject to a thermal law (Coleman-Gurtin or Gurtin-Pipkin thermal law) adding some viscous damping mechanism to the electric field in $x-$direction and $z-$direction, and we discuss several cases. Then, there is no need to control the electrical field components in $x$-direction and $z-$ direction to establish an exponential decay of solutions when the beam is subjected to a Coleman-Gurtin law, otherwise a polynomial stability is established with Gurtin-Pipkin thermal law in case when the electrical field components are damped. 
\end{abstract}
\smallskip
\begingroup
 \hypersetup{linkcolor=blue} 
 \tableofcontents 
\endgroup

\section{Introduction}
\par Piezoeletric beams are structural elements made from materials that exhibit piezoelectric properties. These materials such as  quartz, lead zirconate titante (PZT), polyvinglidene fluoride (PVDF)...etc, are substances that can generate an electric charge when subjected to a mechanical stress, and conversely, they can deform when electric field is applied to them.
This unique property make them ideal for various applications in sensing, actuation, and energy harvesting.  Some of the lattest cutting-edge applications involve: piezoelectric transistor,  structural health monitoring (biogradable piezelectric pressure sensor), energy harvesting (A Shoe Sole), cardiac pacemakers, course-changing bullets...etc.\\
There are mainly three ways to electrically actuate piezoelectric materials: supplying voltage, current or charge to its electrodes. 
  Morris and Özer in \cite{19} considered a fully dynamic magnetic  piezoelectric beam with voltage-driven
electrodes, and they proved that for almost all choices of system parameters
a simple electrical feedback yields strong stability. However,
for other choices of parameters this is not the case.\\
In \cite{2}, Akil  investigated the stabilization of a system of piezoelectric beams under (Coleman or Pipkin)–Gurtin
thermal law with magnetic effect, he proved that the piezoelectric Coleman–Gurtin system is exponentialy stable and the piezoelectric Gurtin–Pipkin system has a polynomial energy decay rate of type $t^{-1}$. \\
Morris and Özer in \cite{20}  used
variational approach  to derive a model for a piezoelectric beam that includes magnetic effects.
It is proven that the partial differential equation model is well-posed. The presence of magnetic effects significantly influences the stability of the control system. For almost all system parameters the piezoelectric
beam can be strongly stabilized, but is not exponentially stable in the energy space. Strong
stabilization is achieved using only electrical feedback. Furthermore, using the same electrical feedback, an exponentially stable closed-loop system can be obtained for a set of system parameters of
zero Lebesgue measure. These results are then compared to those obtained from a beam model devoid of magnetic effects.\\
An et Al. in \cite{6} studied the stability of a piezoelectric beams with magnetic
effects of fractional derivative type and with/ without thermal effects of Fourier’s law; they obtained an
exponential stability by taking two boundary fractional dampings and additional thermal effect.\\
 The biggest disadvantage of voltage-controlled piezoelectric materials is that
there is a huge hysteresis observed between the applied voltage  and the resulting mechanical strain.  A newly developed structure has been considered that is current or charge control that leads to less hysteresis e.g. \cite{10,12}. Then,
 it's worth mentioning the work of  Özer in \cite{21}, the author considered
a piezoelectric beam, that is, an elastic beam covered by
a piezo-electric material on its top and bottom surfaces,
insulated at the edges (to prevent fringing effects), and
connected to an external electric circuit to feed either voltage, current or charge to the electrodes.
 The author introduced a fully-dynamic and electrostatic/quasi-static current-controlled PDE
models obtained by a consistent variational approach. The novelty of the models
is due to the inclusion of electro-magnetic effects through scalar electric and vector
magnetic potentials, using  Lorenz gauge and Coulomb gauge models to provide a well-posed
state-space formulations in the corresponding energy spaces with compatibility conditions. The author proved that this model fails to be asymptotically stable if the material
parameters satisfy certain conditions. To achieve at least asymptotic stability the author
proposed an additional controller. Afilal et Al. in  \cite{1} considered a one-dimensional dissipative system of piezoelectric beams with magnetic effect and localized damping. They proved that the system is exponentialy stable using a damping mechanism acting only on one component and on a small part of the beam.\\
Akil et Al. in \cite{3} worked on the same problem of  Özer in \cite{21} but with viscous damping terms (depending on the parameters a, b and c) acting
on the stretching of the centerline of the beam in $x$- direction, electrical field component in $x$- direction, electrical field component in $z$- direction and magnetic field
component in y-direction.  The authers proved the strong stability as well as the exponential
stability. The terms added to the system with the parameters $a,b$ and $c$ all correspond
to different types of distributed electro-magnetic controllers (We recommend reading  the work of  Özer in \cite{21} for more details). Roughly speaking, they worked on the following system:
\begin{small}
\begin{equation}\label{S1.1}
\left\{\begin{tabular}{ll}
$\rho v_{tt}-\alpha v_{xx}-\gamma (\phi+\eta_{t})_{x}+a v_{t}=0,$ &\;\;in
\;\; $(0,L) \times (0,+\infty)$, \\
$-\xi (\phi_{x}+\theta_{t})_{x}+(\eta_{t}+\phi)-\frac{\gamma}{\epsilon_{3}}v_{x}=0,$& \;\;in $\Omega \times (0,+\infty),$\\
$(\theta_{t}+\phi_{x})_{t}-\frac{\mu}{\xi \epsilon_{3}}(\eta_{x}-\theta)+b(\theta_{t}+\phi_{x})=0,$& \;\; in $\Omega \times (0,+\infty),$\\
$(\eta_{t}+\theta)_{t}-\frac{\mu}{\epsilon_{3}}(\eta_{x}-\theta)_{x}-\frac{\gamma}{\epsilon_{3}}v_{tx}
+c(\eta_{t}+\phi)=0,$& \;\; in $\Omega \times (0,+\infty),$\\
$v(0,t)=\alpha v_{x}(L,t)+\gamma\phi(L,t)+\gamma \eta_{t}(L,t)=0$,&\;\;  $t\in (0,\infty),$\\
$\xi\epsilon_{3} (\theta_{t}+\phi_{x})(0,t)=\xi \epsilon_{3} (\theta_{t}+\phi_{x})(L,t)=0,$& \;\;$ t\in (0,\infty),$\\
$\eta (\theta-\eta_{x})(0,t)=\eta(\theta-\eta_{x})(L,t)=0.$&\;\; $t\in (0,\infty),$
\end{tabular}\right.
\end{equation}
\end{small}
where  $a, b, c \geq 0$. The functions  $v,\theta_{t}+\phi_{x},\eta_{t}+\phi$ and $\theta-\eta_{x}$ represents respectively, the stretching of the
centerline of the beam in x-direction, electrical field component in x-direction, electrical field component in z-direction and magnetic field component in y-direction, and
$\xi= \frac{\epsilon_{1}h_{2}}{12\epsilon_{3}}$ . The natural physical constants $\rho, \alpha, \gamma, \epsilon_{1}, \epsilon_{3}, \mu $ denote the mass density
per unit volume, elastic stiffness, piezoelectric coupling coefficient,  permittivity in x
and z directions, and magnetic permeability respectively. Using Lorenz gauge condition, they reformulated the system to achieve the existence and uniqueness of 
solution. Next, by using general criteria of Arendt–Batty, they proved the strong stability in different cases and they proved that it is sufficient to control the stretching of the
center-line of the beam in x-direction to achieve the exponential stability. Numerical
results are also presented to validate the theoretical result. \\
In our paper, we consider a one dimensional piezoelectric beam coupled with  thermal law (Coleman-Gurtin or Gurtin-Pipkin thermal law), then for $m\in[0,1]$ we introduce our problem as follows:
\begin{small}
\begin{equation}\label{S1.1}
\left\{\begin{tabular}{ll}
$\rho v_{tt}-\alpha v_{xx}-\gamma (\phi+\eta_{t})_{x}+a w_{x}=0,$&\;\;in
 $(0,L) \times (0,+\infty)$, \\
$-\xi (\phi_{x}+\theta_{t})_{x}+(\eta_{t}+\phi)-\frac{\gamma}{\epsilon_{3}}v_{x}=0,$&\;\;in $\Omega \times (0,+\infty),$\\
$(\theta_{t}+\phi_{x})_{t}-\frac{\mu}{\xi \epsilon_{3}}(\eta_{x}-\theta)+b(\theta_{t}+\phi_{x})=0,$&\;\;in  $\Omega \times (0,+\infty),$\\
$(\eta_{t}+\theta)_{t}-\frac{\mu}{\epsilon_{3}}(\eta_{x}-\theta)_{x}
-\frac{\gamma}{\epsilon_{3}}v_{tx}+c(\eta_{t}+\phi)=0,$&\;\; in  $\Omega \times (0,+\infty),$\\
$w_{t}-d(1-m)w_{xx}-dm\int_{0}^{+\infty}g(s)w_{xx}(x,t-s)ds+a v_{xt}=0,$&\;\; in  $\Omega \times (0,+\infty),$\\
$v(0,t)=\alpha v_{x}(L,t)+\gamma\phi(L,t)+\gamma \eta_{t}(L,t)=w(0,t)=w(L,t)=0$,&\;\; $t\in (0,\infty)$,\\
$\xi\epsilon_{3} (\theta_{t}+\phi_{x})(0,t)=\xi \epsilon_{3}(\theta_{t}+\phi_{x})(L,t)=0,$ &\;\;$ t\in (0,\infty),$\\
$\eta (\theta-\eta_{x})(0,t)=\eta(\theta-\eta_{x})(L,t)=0,$&\;\; $ t\in (0,\infty),$
\end{tabular}\right.
\end{equation}
\end{small}
where $a,d>0$ and $b,c\geq0$. We denote by $\varphi_{0}$ the initial past history of $w$, then we have
\begin{equation*}
  w(x,0)=w_{0}(x),\;\;\; w(x,-s)=\varphi_{0}(x,s), \;\;\; \mathrm{for} \;\; s>0, \;\; x\in (0,L).
\end{equation*}
The kernel function $g:[0,+\infty[\rightarrow [0,+\infty[$ is a convex integrable function (non-increasing and vanishing at infinity), taking the form
\begin{equation*}
  g(s)=\int_{s}^{\infty}\sigma(r)dr, \;\; s\geq0,
\end{equation*}
where $\sigma:(0,\infty)\rightarrow [0,\infty)$ is a memory kernel $\sigma$  required to satisfy the following hypothesis
\begin{equation*}\label{C10}
\left\{\begin{tabular}{ll}
$\sigma \in L^{1}(0,\infty)\cap \mathcal{C}^{1}(0,\infty)$ with $\int_{0}^{\infty}\sigma(s)ds=g(0)>0,$&\\
$\sigma(0)=\lim\limits_{s\rightarrow 0}\sigma(s)<\infty,$&\\ 
$\sigma$ satisfies the Dafermos condition $\sigma^{\prime}(s)\leq -d_{\sigma}\sigma(s).$&
\end{tabular}\right.(\mathrm{\mathbf{H}})
\end{equation*}
Speaking on Non-Fourier heat flux laws, we mention some relevant  works; we cite Akil, M. et al (\cite{4}) in 2023,  An, Y. et al. (\cite{6}) in  2021,  Dell’Oro, F. (\cite{13}) in 2014 and (\cite{14}) in 2021,  V.V. Chepyzhov et al. (\cite{9}) in 2006;  Zhang, Q. et al. (\cite{25}) in 2014, Eremenko, A. et al. (\cite{11}) in 2011, Giorgi, C. et al. (\cite{15}), in 2001. We add some works for models with memory in past history framework as Akil, M et al. in (\cite{5}), in 2021; Liu, Z et al. (\cite{18}), in 1996 and Rivera JEM et al. (\cite{24}) in 2008. \\
\par As already mentioned, in this article, we study the stabilization of the piezoelectric beam with Coleman-Gurtin or Gurtin-Pipkin thermal law and under Lorenz gauge condition, discussing several cases depending  on the damping mechanism of the electrical field components. The aim of this study is to show that without viscous dampings on the electrical field components, the semigroup associated to the system of equations (\ref{S1.1}) with Coleman-Gurtion thermal law still exponentially stable, then it is suffices to consider the thermic damping of type Coleman-Gurtin to establish an exponential decay of solutions.  In the second section, we give a reformulation of the system (\ref{S1.1}) into a well-posed system, using a Lorenz gauge condition. In the third section, we give an exponential stability result of the system, in case the piezoelectric beam is only subjected to a Coleman-Gurtin thermal law. A polynomial stability of rate $t^{-1}$ is discussed in the last section in case of Gurtin-Pipkin law. We finish the paper with a conclusion.
\section{Reformulation and well-Posedness of the problem}
In this section, we aim to decouple the equations of electromagnetism in the stretching system by 
employing a specific gauge condition, this allows us to reformulate the problem as a well-posed problem. 
One commonly used gauge condition is the Lorenz gauge, which can be expressed by the following equation:
\begin{equation}\label{C1}
  -\xi \theta_{x}+\eta=\frac{\xi \epsilon_{3}}{\mu}\phi_{t},
\end{equation}
with the boundary conditions
\begin{equation*}
  \theta(0,t)=\theta(L,t)=0.
\end{equation*}
In this case, the terms $-\xi\theta_{tx}+\eta_{t}$ in $(\ref{S1.1})_{2}$ transformed into $\frac{\xi \epsilon_{3}}{\mu}\phi_{tt}$, as well, the terms $\phi_{tx}-\frac{\mu}{\xi\epsilon_{3}}(\eta_{x}-\theta)$ and $\phi_{t}-\frac{\eta}{\epsilon_{3}}(\eta_{x}-\theta)_{x}$ in $(\ref{S1.1})_{3}$ and $(\ref{S1.1})_{4}$ are transformed into $-\frac{\mu}{\xi\epsilon_{3}}(\xi\theta_{xx}-\theta)$ and $-\frac{\mu}{\xi \epsilon_{3}}(\xi \eta_{xx}-\eta)$, respectively.  \\
On the other side, we reformulate the problem (\ref{S1.1}) using the history framework of Dafermos. To this end, for $s>0$, we consider the auxiliary function
\begin{equation*}
  \kappa(x,s)=\int_{0}^{s}w(x,t-r)dr, \;\; x\in (0,L), \;\; s >0.
\end{equation*}
Then, for $m \in [0,1]$ we can write the system (\ref{S1.1}) into the following form
 \begin{small}
 \begin{equation}\label{C2}
\left\{\begin{tabular}{ll}
$\rho v_{tt}-\alpha v_{xx}-\gamma (\phi+\eta_{t})_{x}+a w_{x}=0,$ &\;\;in
 $(0,L) \times (0,+\infty)$, \\
$\phi_{tt}-\frac{\mu}{\epsilon_{3}}\phi_{xx}+\frac{\mu}{\xi \epsilon_{3}}\phi-\frac{\gamma \mu }{\xi \epsilon_{3}^{2}}v_{x}=0,$&\;\; in $(0,L) \times (0,+\infty),$\\
$\theta_{tt}-\frac{\mu}{\epsilon_{3}}\theta_{xx}+\frac{\mu}{\xi \epsilon_{3}}\theta+b(\theta_{t}+\phi_{x})=0,$&\;\; in $(0,L) \times (0,+\infty),$\\
$\eta_{tt}-\frac{\mu}{\epsilon_{3}}\eta_{xx}+\frac{\mu}{\xi \epsilon_{3}}\eta-\frac{\gamma}{\epsilon_{3}}v_{tx}+c(\eta_{t}+\phi)=0,$&\;\; in $(0,L) \times (0,+\infty),$\\
$w_{t}-d(1-m)w_{xx}-dm\int_{0}^{+\infty}\sigma(s)\kappa_{xx}(s)ds+a v_{xt}=0,$&\;\; in $(0,L) \times (0,+\infty),$\\
$\kappa_{t}+\kappa_{s}-w=0,$\;\; in $(0,L) \times (0,+\infty)$,\\
$v(0,t)=\alpha v_{x}(L,t)+\gamma\phi(L,t)+\gamma \eta_{t}(L,t)=w(0,t)=w(L,t)=0$,&\;\;  $t\in (0,\infty),$\\
$\phi_{x}(0,t)=\phi_{x}(L,t)=\eta_{x}(0,t)=\eta_{x}(L,t)=\theta(0,t)=\theta(L,t)=0$,& \;\;$ t\in (0,\infty),$\\
$(v,\phi,\theta,\eta,v_{t},\phi_{t},\theta_{t},\eta_{t})(\cdot,0)
=(v^{0},\phi^{0},\theta^{0},\eta^{0},v^{1},\phi^{1},\theta^{1},\eta^{1}),$&\;\;
$x\in (0,L).$
\end{tabular}\right.
\end{equation}
\end{small}
We give in the following lemma the energy functional of the system (\ref{C2}):
\begin{lem}
  The energy associated to the system (\ref{C2}) is given by the following sum of energies
  \begin{equation*}
  E(t)=E_{k}(t)+E_{p}(t)+E_{B}(t)+E_{elec}(t)+E_{m}(t),
  \end{equation*}
  where
  \begin{equation}
  \begin{tabular}{ll}
  $E_{k}(t)=\frac{\rho}{2}\int_{0}^{L}|v_{t}|^{2}dx$,\;\; $E_{p}(t)=\frac{\alpha}{2}\int_{0}^{L}|v_{x}|^{2}dx$,\;\; $E_{B}(t)=\frac{\mu}{2}\int_{0}^{L}|\theta-\eta_{x}|^{2}dx,$&\\
  $E_{elec}(t)=\frac{1}{2}\int_{0}^{L}[\xi \epsilon_{3}|\theta_{t}+\phi_{x}|^{2}+\epsilon_{3}|\eta_{t}+\phi|^{2}]dx$,&\\
   $E_{m}(t)=\frac{1}{2}\int_{0}^{L}|w|^{2}dx+\frac{md}{2}\int_{0}^{\infty}\int_{0}^{L}\sigma(s)|\eta_{x}|^{2}dsdx$.
  \end{tabular}
  \end{equation}
  and 
  \begin{eqnarray*}
\frac{d}{dt}E(t)&=&-b\xi\epsilon_{3}\int_{0}^{L}|\theta_{t}+\phi_{x}|^{2}dx
-c\epsilon_{3}\int_{0}^{L}|\eta_{t}+\phi|^{2}dx+(m-1)d\int_{0}^{L}|w_{x}|^{2}dx\\
&&+md\int_{0}^{\infty}\int_{0}^{L}\sigma^{\prime}(s)|\kappa_{x}|^{2}dxds.
\end{eqnarray*}
\end{lem}
\begin{proof}
Multiplying $(\ref{C2})_{1}$ by $v_{t}$, integrating by parts over $(0,L)$, then taking the real part, we obtain
\begin{equation}\label{C9}
\frac{d}{dt}E_{k}(t)+\frac{d}{dt}E_{p}(t)+\Re(\gamma\int_{0}^{L}(\phi+\eta_{t})\overline{v}_{xt}dx)+a \Re(\int_{0}^{L}w_{x}\overline{v}_{t}dx)=0.
\end{equation}
Multiplying $(\ref{C2})_{3}$, $(\ref{C2})_{4}$ by $\xi \epsilon_{3}(\overline{\theta_{t}+\phi_{x}})$, and $\epsilon_{3}(\overline{\eta_{t}+\phi})$, respectively, then operating as in Akil et Al.  (\cite{3}), we get
\begin{eqnarray}\label{C7}
  \frac{\xi \epsilon_{3}}{2}\frac{d}{dt}\int_{0}^{L}|\theta_{t}+\phi_{x}|^{2}dx +\frac{\mu}{2}\frac{d}{dt}\int_{0}^{L}|\theta|^{2}dx\notag\\
  -\Re(\mu\int_{0}^{L}\eta_{x}\overline{\theta}_{t}dx)-\Re(\mu\int_{0}^{L}\eta_{x}\overline{\phi}_{x}dx)\notag\\
  +\Re(\mu\int_{0}^{L}\theta \overline{\phi}_{x}dx)+b\xi \epsilon_{3}\int_{0}^{L}|\theta_{t}+\phi_{x}|^{2}dx&=& 0
\end{eqnarray}
and
\begin{eqnarray}\label{C8}
\frac{\epsilon_{3}}{2}\frac{d}{dt}\int_{0}^{L}|\eta_{t}+\phi|^{2}dx+\frac{\mu}{2}\frac{d}{dt}\int_{0}^{L}|\eta_{x}|^{2}dx
-\Re(\mu \int_{0}^{L}\eta \overline{\eta}_{xt}dx)\notag\\
+\Re(\mu \int_{0}^{L}\eta_{x}\overline{\phi}_{x}dx)
-\Re(\mu \int_{0}^{L}\theta \overline{\phi}_{x}dx)\notag\\
-\Re(\gamma \int_{0}^{L}v_{xt}(\overline{\phi+\eta_{t}})dx)+c\epsilon\int_{0}^{L}|\eta_{t}+\phi|^{2}dx&=&0.
\end{eqnarray}
Now, multiplying $(\ref{C2})_{5}$ by $\overline{w}$, integrating by parts over $(0,L)$, we get
\begin{eqnarray}\label{C5}
  \frac{1}{2}\frac{d}{dt}\int_{0}^{L}|w|^{2}dx+(1-m)d\int_{0}^{L}|w_{x}|^{2}dx\notag\\
  +\Re(md\int_{0}^{L}\int_{0}^{\infty}\sigma(s)\kappa_{x}(s)\overline{w}_{x}dsdx)
  -a\int_{0}^{L}v_{t}\overline{w}_{x}dx&=&0.
\end{eqnarray}
We differentiate the equation $(\ref{C2})_{6}$ with respect to $x$, we get
\begin{equation}\label{C3}
  \kappa_{xt}+\kappa_{xs}-w_{x}=0.
\end{equation}
Multiplying  (\ref{C3}) by $dm\sigma(s)\overline{\kappa}_{x}$, integrating over  $(0,L)\times (0,\infty)$, we obtain
\begin{eqnarray}\label{C4}
\frac{1}{2}\frac{d}{dt}md\int_{0}^{L}\int_{0}^{\infty}\sigma(s)|\kappa_{x}|^{2}dsdx\notag\\
-\frac{md}{2}\int_{0}^{L}\int_{0}^{\infty}\sigma^{\prime}(s)|\kappa_{x}|^{2}dsdx
&=&\Re(md\int_{0}^{L}\int_{0}^{\infty}\sigma(s)\kappa_{x}\overline{w}_{x}dsdx).
\end{eqnarray}
Inserting (\ref{C3}) in (\ref{C5}), we obtain
\begin{equation}\label{C6}
\frac{d}{dt}E_{m}(t)-a\int_{0}^{L}v_{t}\overline{w}_{x}dx=(m-1)d\int_{0}^{L}|w_{x}|^{2}dx
+\frac{md}{2}\int_{0}^{L}\int_{0}^{\infty}\sigma^{\prime}(s)|\kappa_{x}|^{2}dsdx.
\end{equation}
Summing the result (\ref{C9}), (\ref{C7}), (\ref{C8}), (\ref{C5}) and (\ref{C6}), we obtain
\begin{eqnarray*}
\frac{d}{dt}E(t)&=&-b\xi\epsilon_{3}\int_{0}^{L}|\theta_{t}+\phi_{x}|^{2}dx
-c\epsilon_{3}\int_{0}^{L}|\eta_{t}+\phi|^{2}dx+(m-1)d\int_{0}^{L}|w_{x}|^{2}dx\\
&&+md\int_{0}^{\infty}\int_{0}^{L}\sigma^{\prime}(s)|\kappa_{x}|^{2}dxds.
\end{eqnarray*}
Thus the proof has been completed.
\end{proof}
Now, we define the following state as
\begin{equation*}
U=(v,z,u^{1},u^{2},u^{3},w,\kappa)^\top,
\end{equation*}
with $z=v_{t}$, $u^{1}=\theta-\eta_{x}$, $u^{2}=\theta_{t}+\phi_{x}$ and $u^{3}=\eta_{t}+\phi$.
We make use of Lorenz gauge condition and $(\ref{C2})_{2}$, we write the following compatibility condition:
\begin{equation*}
  \xi u^{2}_{x}-u^{3}+\frac{\gamma}{\epsilon_{3}}v_{x}=0.
\end{equation*}
We define the $\sigma-$ weighted space  $W$ as follow
\begin{equation*}
  W=L^{2}_{\sigma}(\mathbb{R}^{+};H^{1}_{0}(0,L)).
\end{equation*}
endowed with the inner product
\begin{equation*}
  (\kappa_{1},\kappa_{2})_{W}=md\int_{0}^{L}\int_{0}^{\infty}\sigma(s)\kappa_{1x}\overline{\kappa}_{2x}dsdx, \;\; \forall \kappa_{1},\kappa_{2} \in W.
\end{equation*}
Now, we define the following Hilbert space $\mathcal{H}$ as
\begin{eqnarray*}
  \mathcal{H}&=&\{U\in (H^{1}_{L}(0,L)\times (L^{2}(0,L))^{5}\times W), u_{x}^{2}\in L^{2}(0,L), u^{2}(0)=u^{2}(L)=0,\\
&&\xi u^{2}_{x}-u^{3}+\frac{\gamma}{\epsilon_{3}}v_{x}=0\},
\end{eqnarray*}
with the following inner product for $U, W \in \mathcal{H}$
\begin{eqnarray*}
\left\langle U, W \right\rangle_{\mathcal{H}}
 &=& \int_{0}^{L}(\alpha v_{x}\overline{w^{1}_{x}}+\rho z^{1}\overline{w^{2}}+\mu
 u^{1}\overline{w^{3}}+\xi \epsilon_{3}u^{2}\overline{w^{4}}+\epsilon_{3}u^{3}\overline{w^{5}}+w\overline{w^{6}})dx
+(\kappa, w^{7})_{W}.
\end{eqnarray*}
 For $m \in [0,1]$, we define the unbounded linear operator $\mathcal{A}$ as
\begin{eqnarray}\label{A1}
\mathcal{A}U=
 \left [\begin{array}{c}
  z\\
\frac{\alpha}{\rho}v_{xx}+\frac{\gamma}{\rho}u_{x}^{3}-\frac{a}{\rho}w_{x}\\
 u^{2}-u_{x}^{3}\\
   -\frac{\mu }{\xi \epsilon_{3}}u^{1}-bu^{2}\\
    -\frac{\mu}{\epsilon_{3}}u^{1}_{x}+\frac{\gamma}{\epsilon_{3}}z_{x}-cu^{3}\\
    d\Lambda^{m}_{xx}-az_{x}\\
    -\kappa_{s}+w
\end{array}
\right ]
\end{eqnarray}
in a domain $\mathcal{D}(\mathcal{A})\subset \mathcal{H}$ defined as,
\begin{eqnarray}\label{A2}
\mathcal{D}(\mathcal{A})&:=& \{ U:=(v,z,u^{1},u^{2},u^{3},w,k)\in \mathcal{H} \\ \notag
&& z\in H^{1}_{L}(0,L),\;\; v \in H^{2}(0,L)\cap H^{1}_{L}(0,L),\;\; u^{1},u^{2}\in H^{1}_{0}(0,L), \\ \notag
&&u^{3}\in H^{1}(0,L),\;\;w \in H^{1}_{0}(0,L),\;\;  \alpha v_{x}(L)+\gamma u^{3}(L)=0,\\ \notag
&&\Lambda^{m}_{x}\in H^{1}(0,L), \; \kappa_{s}\in W, \; \kappa(.,0)=0\}
\end{eqnarray}
with $\Lambda^{m}=(1-m)w+m\int_{0}^{\infty}\sigma(s)\kappa(s)ds$.
\begin{prop}\label{PROPO1}
For $m\in [0,1]$, under the hypothesis $(\mathrm{\mathbf{H}})$ the unbounded linear operator $\mathcal{A}$  is the infinitesimal generator of a $\mathcal{C}_{0}$-semigroup of contractions on the Hibert space $\mathcal{H}$.
\end{prop}
\begin{proof}
The unbounded linear operator $\mathcal{A}$ is dissipative for $m\in [0,1]$: Under the hypothesis $(\mathrm{\mathbf{H}})$ and for $U\in \mathcal{D}(\mathcal{A})$, we claim that 
\begin{eqnarray*}
  \Re\langle \mathcal{A}U,U\rangle&=&-b\xi \epsilon_{3}\int_{0}^{L}|u^{2}|^{2}dx-c\epsilon_{3}\int_{0}^{L}|u^{3}|^{2}dx\\
  &&+(m-1)d\int_{0}^{L}|w_{x}|^{2}dx+\frac{md}{2}\int_{0}^{L}\int_{0}^{L}\sigma^{\prime}(s)|\kappa_{x}|^{2}dsdx\leq 0.
\end{eqnarray*}
The unbounded linear operator $-\mathcal{A}$ is surjective.\\
 Indeed, for $F:=(f^{1},f^{2},f^{3},f^{4},f^{5},f^{6},f^{7})^\top\in \mathcal{H}$, we search for $U\in \mathcal{D}(\mathcal{A})$ solution of
\begin{eqnarray}
-z=f^{1},\label{B1}\\
-\alpha v_{xx}-\gamma u_{x}^{3}+aw_{x}=\rho f^{2},\label{B2}\\
 -u^{2}+u_{x}^{3}=f^{3},\label{B3}\\
 \frac{\mu }{\xi \epsilon_{3}}u^{1}+bu^{2}=f^{4}, \label{B4}\\
  \frac{\mu}{\epsilon_{3}}u^{1}_{x}-\frac{\gamma}{\epsilon_{3}}z_{x}+cu^{3}=f^{5}, \label{B5}\\
 -d\Lambda^{m}_{xx}+az_{x}=f^{6},\label{B6}\\
 \kappa_{s}-w=f^{7}.\label{B7}
\end{eqnarray}
From (\ref{B1}), we have
\begin{equation}
z=-f^{1} \;\;\;  and \;\;\;  z\in H^{1}_{L}(0,L).
\end{equation}
From (\ref{B7}), we have
\begin{equation}\label{B8}
\kappa(\cdot, s)=sw+\int_{0}^{s}f^{7}(x,\tau)d\tau.
\end{equation}
Then, from (\ref{B1}) and (\ref{B6}) we obtain
\begin{equation}
\Lambda^{m}_{xx}=-d^{-1}(f^{6}+af_{x}^{1}),
\end{equation}
it follows
\begin{equation}\label{B9}
\Lambda^{m}=-d^{-1}\int_{0}^{x}\int_{0}^{x_{1}}(f^{6}+af_{x_{2}}^{1})dx_{2}dx_{1}+d^{-1}\dfrac{x}{L}\int_{0}^{L}\int_{0}^{x_{1}}(f^{6}+af_{x_{2}}^{1})dx_{2}dx_{1}
\end{equation}
then, we conclude $w$ easily from (\ref{B8}) and (\ref{B9}) defined as 
\begin{equation*}
w=\dfrac{1}{\tilde{m}}\left[\Lambda^{m}(x)-m\int_{0}^{\infty}\sigma(s)\int_{0}^{s}f^{7}(\cdot,\tau)d\tau ds\right]
\end{equation*}
where $\tilde{m}=(1-m)d+md \int_{0}^{\infty} s\sigma(s)ds >0$. Then, we conclude $k(\cdot,s)$ easily from the expression (\ref{B8}).\\
We seek for $u^{1},u^{2},u^{3}$ and $v$ by considering several cases, using Green's function theory, we obtain $U\in \mathcal{D}(\mathcal{A})$ by distinguishing different cases as follow \\
\textbf{Case 1. $(b,c)= (0,0)$}\\
\begin{small}
 \begin{eqnarray*}
 u^{1}(x)&=& \dfrac{ \xi \epsilon_{3}}{\mu}f^{4}(x),\\
 u^{2}(x)&=&\frac{1}{2\sqrt{\alpha_{1}}}\frac{e^{\sqrt{\frac{\alpha_{1}}{\xi \alpha}}x}-e^{-\sqrt{\frac{\alpha_{1}}{\xi \alpha}}x}}{e^{\sqrt{\frac{\alpha_{1}}{\xi \alpha}}L}-e^{-\sqrt{\frac{\alpha_{1}}{\xi \alpha}}L}}\times \notag\\
  &&\left[ \int_{0}^{L}e^{\sqrt{\frac{\alpha_{1}}{\xi \alpha}}(x_{1}-L)}F(x_{1})dx_{1}-\int_{0}^{L}e^{-\sqrt{\frac{\alpha_{1}}{\xi \alpha}}(x_{1}-L)}F(x_{1})dx_{1} \right] \notag\\
  &&+\frac{1}{2\sqrt{\alpha_{1}}}\left[ \int_{0}^{L}e^{\sqrt{\frac{\alpha_{1}}{\xi \alpha}}(x_{1}-x)}F_{1}(x_{1})dx_{1}-\int_{0}^{L}e^{-\sqrt{\frac{\alpha_{1}}{\xi \alpha}}(x_{1}-x)}F_{1}(x_{1})dx_{1} \right],
 \end{eqnarray*}
 \end{small}
 for $\alpha_{1}=\alpha +\frac{\gamma^{2}}{\epsilon_{3}}$ and $F_{1}(x)=-\alpha_{1}f^{3}(x)+\dfrac{\rho \gamma}{\epsilon_{3}}f^{2}(x)-a\dfrac{\gamma}{\epsilon_{3}}w_{x}(x)$. Then, we obtain
\begin{small}
 \begin{equation*}
 v(x)=\rho \int_{0}^{x}\int_{0}^{x_{1}}f^{2}(x_{2})dx_{2}dx_{1}-a\int_{0}^{x}w(x_{1})dx_{1}+\gamma \xi u^{2}(x)+\rho x \int_{0}^{L}f^{2}(x_{1})dx_{1},
 \end{equation*}
 \end{small}
 then we conclude $u^{3}$ from the compatibility condition.\\
 \textbf{Case 2. $b=0$ and $c \neq 0$}\\
 For $x\in (0,L)$, we have
 \begin{small}
 \begin{eqnarray*}
 u^{1}(x)&=&\frac{\xi \epsilon_{3}}{\mu}f^{4}(x),\\
 u^{3}(x)&=& c^{-1}f^{5}(x)-c^{-1}\left(\frac{\mu}{\epsilon_{3}}u^{1}_{x}(x)+\frac{\gamma}{\epsilon_{3}}f^{1}_{x}(x)\right),\\
 v(x)&=&\dfrac{a}{\alpha}\int_{0}^{x}w(x_{1})dx_{1}-\dfrac{\rho}{\alpha} \int_{0}^{x}\int_{0}^{x_{1}}f^{2}(x_{2})dx_{2}dx_{1}-\dfrac{1}{c\alpha}\int_{0}^{x}f^{5}(x_{1})dx_{1}\\
 &&+ \dfrac{\xi}{\alpha c}\int_{0}^{x}f^{4}(x_{1})dx_{1}
 +\dfrac{\gamma}{\alpha\epsilon_{3}c}f^{1}(x)+\rho\dfrac{x}{\alpha}\int_{0}^{L}f^{2}(x_{2})dx_{2},\\
 u^{2}(x)&=&\frac{1}{c\xi}\int_{0}^{x}f^{5}(x_{1})dx_{1}-\frac{1}{c\xi}\left(\xi f^{4}(x)+\frac{\gamma}{\epsilon_{3}}f^{1}(x)\right)-\frac{\gamma}{\epsilon_{3}}v(x).
 \end{eqnarray*}
 \end{small}
\textbf{Case 3. $b\neq0$ and $c=0$}\\ For $x\in (0,L)$, we have
\begin{small}
 \begin{eqnarray*}
 u^{1}(x)&=&\dfrac{\epsilon_{3}}{\mu}\int_{0}^{x}f^{5}(x_{1})dx_{1}-\dfrac{\gamma}{\mu}f^{1}(x),\\
 u^{2}(x)&=&b^{-1}\left[f^{4}(x)-\dfrac{1}{\xi}\int_{0}^{x}f^{5}(x_{1})dx_{1}+\dfrac{\gamma}{\xi\epsilon_{3}}f^{1}(x)\right],\\
 v(x)&=&\alpha^{-1}\left[ a\int_{0}^{x}w(x_{1})dx_{1}-\int_{0}^{x}\int_{0}^{x_{1}}\left( \rho f^{2}(x_{2})+\gamma u^{2}(x_{2})+\gamma f^{3}(x_{2})\right)dx_{2}dx_{1}\right]\\
 &&+\rho \alpha^{-1}x\int_{0}^{L}f^{2}(x_{1})dx_{1},\\
 u^{3}(x)&=&\gamma^{-1}aw(x)-\gamma^{-1}\alpha v_{x}(x)-\rho \int_{0}^{x}f^{2}(x_{1})dx_{1}+\gamma^{-1}\rho \int_{0}^{L}f^{2}(x_{1})dx_{1}.
 \end{eqnarray*}
 \end{small}
 \textbf{Case 4. $(b,c)\neq(0,0)$}\\For $x\in (0,L)$, we have
 \begin{small}
 \begin{equation*}
 u^{1}(x)=-\int_{0}^{x}e^{-\frac{c}{b\xi}(x_{1}-x)}F_{2}(x_{1})dx_{1},
 \end{equation*}
 \end{small}
 with $F_{2}(x)=\dfrac{\epsilon_{3}c}{\mu}f^{4}(x)+\dfrac{bc\epsilon_{3}}{\mu}f^{3}(x)-\dfrac{b\epsilon_{3}}{\mu}f^{5}_{x}(x)
 +\dfrac{b\gamma}{\mu}f^{1}_{xx}(x)$.
 \begin{small}
 \begin{eqnarray*}
 u^{2}(x)&=&b^{-1}\left[f^{4}(x)+\dfrac{\mu}{\xi\epsilon_{3}}\int_{0}^{x}e^{-\frac{c}{b\xi}(x_{1}-x)}F_{2}(x_{1})dx_{1}\right],\\
 v(x)&=& \frac{1}{\alpha_{2}}\int_{0}^{x}w(x_{1})dx_{1}-\frac{\gamma}{\alpha_{2}}u^{2}(x)-\frac{\rho}{\alpha_{2}} \int_{0}^{x}\int_{0}^{x_{1}}f^{2}(x_{2})dx_{2}dx_{1}+\frac{x}{\alpha_{2}}\int_{0}^{L}f^{2}(x_{1})dx_{1}, 
 \end{eqnarray*}
 \end{small}
 for $\alpha_{2}=\alpha+\frac{\gamma}{\epsilon_{3}}$. Then, we can conclude  $u^{3}$ using compatibility condition.\\
\end{proof}
Thus, the well-posedness result is giving in the following theorem, relying on the result of the proposition \ref{PROPO1}:
\begin{thm}
For $m\in[0,1]$, under the hypothesis $(\mathrm{\mathbf{H}})$ and for any initial datum $U_{0}\in \mathcal{H}$ there exists a unique solution $U$ of the
problem (\ref{S1.1}) such that
\begin{equation*}
  U\in C([0,+\infty ),\mathcal{H}).
\end{equation*}
Moreover, if $U_{0}\in \mathcal{D}(\mathcal{A})$, then
\begin{equation*}
  U\in C([0,+\infty ),\mathcal{D}(\mathcal{A}))\cap C^{1}([0,+\infty),
\mathcal{H}).
\end{equation*}
\end{thm}
\begin{proof}
The proof of this theorem rely on the application of Lumer-Philips used in  proposition \ref{PROPO1}.
\end{proof}
\section{Exponential stability result of the piezoelectric beam with Coleman-Gurtin thermal law}
The objective of this section is to analyse the exponential decay of solution for the piezoelectric beam with Coleman-
Gurtin law ($m\in (0,1)$) distinguishing four cases. We start with the following case:\\
\subsection{The electrical field components in $x$- and $z$- direction are damped $\mathbf{(b,c)\neq (0,0)}$}
We give the following result considering a piezoelectric beam with Coleman-Gurtin thermal 
law and the electrical field components are damped:
\begin{thm}\label{THEOREM1}
For $m\in (0,1)$, under the hypothesis $(\mathrm{\mathbf{H\psi}})$, the $\mathcal{C}_{0}-$semigroup of contractions $(e^{t\mathcal{A}})_{t\geq0}$ is exponentially stable, i.e. there exists $M\geq 1$ and $\epsilon >0$ such  that
\begin{equation*}
  \|e^{t\mathcal{A}}U_{0}\|_{\mathcal{H}}\leq Me^{-\epsilon t} \|U_{0}\|_{\mathcal{H}}, \;\; \mathrm{for}  \;\; t \geq 0.
\end{equation*}
\end{thm}
Following Hung and Pr\"{u}ss in \cite{17,23}, we have to check the following conditions:
\begin{equation*}
i\mathbb{R}\subset \rho(\mathcal{A}) \;\;\;\;\;\;\;\;\;\;\;\;\;\;\;\;\;\; \;\;\;\;\;\;\;\;\;\;\;\;\;\;\;\;\;\;\;\;\;\;\;\;\;\;\;\;\;\;(\mathrm{\mathbf{C1}})
\end{equation*}
and
\begin{equation*}
\sup_{\lambda \in \mathbb{R}}\|(i\lambda I-\mathcal{A})^{-1}\|_{\mathcal{L}(\mathcal{H})}=o(1). \;\;\;\;\;\;\;\;\;\;\;\;\;\;\;\;\;\; (\mathrm{\mathbf{C2}})
\end{equation*}
  Let $(\lambda,U=(v,z,u^{1},u^{2},u^{3},w,\kappa)^\top)\in \mathbb{R}^{\ast}\times\mathcal{D}(\mathcal{A})$, with $|\lambda|\geq1$, such that
  \begin{equation}\label{CF12}
  (i\lambda I- \mathcal{A})U=F:=(f^{1},f^{2},f^{3},f^{4},f^{5},f^{6},f^{7})^\top\in \mathcal{H}.
  \end{equation}
  We detail as
  \begin{eqnarray}
   i\lambda v-z &=& f^{1}, \label{CA3} \\
    i\lambda z-\frac{\alpha}{\rho}v_{xx}-\frac{\gamma}{\rho}u^{3}_{x}+\frac{a}{\rho}w_{x} &=& f^{2},\label{CA4} \\
    i\lambda u^{1}-u^{2}+u^{3}_{x} &=& f^{3},\label{CA5} \\
    i\lambda u^{2}+\frac{\mu}{\xi \epsilon_{3}}u^{1}+bu^{2} &=& f^{4}, \label{CA6}\\
    i\lambda u^{3}+\frac{\mu}{\epsilon_{3}}u^{1}_{x}-\frac{\gamma}{\epsilon_{3}}z_{x}+cu^{3} &=& f^{5}, \label{CA7}\\
    i\lambda w-\lambda \Lambda^{m}_{xx}+az_{x} &=& f^{6}, \label{CA8}\\
    i\lambda k+k_{s}-w &=& f^{7}.\label{CA9}
  \end{eqnarray}
  
The proof of this theorem  is given through the use of  several lemmas, we start with the following lemma:
\begin{lem}\label{CA10}
For $m\in (0,1)$, assume that the hypothesis  $(\mathrm{\mathbf{H}})$ holds.
 The solution  $U\in \mathcal{D}(\mathcal{A})$ satisfy the following estimates
\begin{eqnarray}
\int_{0}^{L}|w_{x}|^{2}dx&\leq&\mathcal{S}_{1}\|U\|_{\mathcal{H}}\|F\|_{\mathcal{H}}, \label{CA24}\\
\int_{0}^{L}|w|^{2}dx&\leq&\mathcal{S}_{2}\|U\|_{\mathcal{H}}\|F\|_{\mathcal{H}}, \label{CA25}\\
\int_{0}^{L}|u^{2}|^{2}dx&\leq&\mathcal{S}_{3}\|U\|_{\mathcal{H}}\|F\|_{\mathcal{H}},\label{CA26}\\
\int_{0}^{L}|u^{3}|^{2}dx&\leq&\mathcal{S}_{4}\|U\|_{\mathcal{H}}\|F\|_{\mathcal{H}},\label{CA27}
\end{eqnarray}
and
\begin{eqnarray}
  \int_{0}^{+\infty} \int_{0}^{L}\sigma(s)|\kappa_{x}|^{2}dxds&\leq& \mathcal{S}_{5}\|U\|_{\mathcal{H}}\|F\|_{\mathcal{H}}, \label{CA28}\\
  \int_{0}^{L}|\Lambda_{x}^{m}|^{2}dx &\leq& \mathcal{S}_{6}\|U\|_{\mathcal{H}}\|F\|_{\mathcal{H}},\label{CA29}
\end{eqnarray}
with
\begin{eqnarray*}
\mathcal{S}_{1}&=&\dfrac{1}{(1-m)d},\;\; \mathcal{S}_{2}=\dfrac{C_{P}}{(1-m)d},\;\; 
\mathcal{S}_{3}=\dfrac{1}{b\xi\epsilon_{3}},\\
\mathcal{S}_{4}&=&\dfrac{1}{c\epsilon_{3}},\;\; \mathcal{S}_{5}=\frac{2}{mdd_{\sigma}},\;\; \mathcal{S}_{6}=\dfrac{2(1-m)}{d}+\dfrac{4g(0)}{mdd_{\sigma}}.
\end{eqnarray*}
\end{lem}
\begin{proof}
Taking the inner product of (\ref{CF12}) with $U\in \mathcal{D}(\mathcal{A})$, we find
\begin{eqnarray*}
b\xi \epsilon_{3}\int_{0}^{L}|u^{2}|^{2}dx
+c\epsilon_{3}\int_{0}^{L}|u^{3}|^{2}dx +(1-m)d\int_{0}^{L}|w_{x}|^{2}dx \notag\\
-\frac{md}{2}\int_{0}^{L}\int_{0}^{\infty}\sigma^{\prime}(s)|\kappa_{x}|^{2}dsdx=-\Re(\mathcal{A}U,U)_{\mathcal{H}}&\leq& \|F\|_{\mathcal{H}}\|U\|_{\mathcal{H}}.
\end{eqnarray*}
Using the hypothesis $(\mathcal{\mathbf{H}})$, we get the following estimation
\begin{eqnarray*}
b\xi \epsilon_{3}\int_{0}^{L}|u^{2}|^{2}dx
c\epsilon_{3}\int_{0}^{L}|u^{3}|^{2}dx+(1-m)d\int_{0}^{L}|w_{x}|^{2}dx \notag\\
+\frac{dd_{\sigma}m}{2}\int_{0}^{L}\int_{0}^{\infty}\sigma(s)|\kappa_{x}|^{2}dsdx=-\Re(\mathcal{A}U,U)_{\mathcal{H}}&\leq& \|F\|_{\mathcal{H}}\|U\|_{\mathcal{H}},
\end{eqnarray*}
then since $m\in (0,1)$ the inequalities (\ref{CA24}), (\ref{CA26}), (\ref{CA27}) and  (\ref{CA28})
 are established. Making use of Poincaré's inequality in the result (\ref{CA24}) we obtain (\ref{CA25}). On the other side,
   the inequality $2ab\leq a^{2}+b^{2}$ and Cauchy-Schwartz inequality leads to
\begin{eqnarray*}
|\Lambda_{x}|^{2}&\leq& 2(1-m)^{2}|w_{x}|^{2}+ 2\left(\int_{0}^{+\infty}|\sigma(s)||\kappa_{x}(s)|ds\right)^{2}\\
&\leq& 2 (1-m)|w_{x}|^{2}+2g(0)\int_{0}^{+\infty}\sigma(s)|\kappa_{x}(s)|^{2}ds,
\end{eqnarray*}
then,
\begin{eqnarray*}
\int_{0}^{L}|\Lambda_{x}|^{2}dx&\leq&2(1-m)^{2}\int_{0}^{L}|w_{x}|^{2}dx
+2g(0)\int_{0}^{L}\int_{0}^{+\infty}\sigma(s)|\kappa_{x}(s)|^{2}dsdx\\
&\leq&\left( 2(1-m)^{2}\mathcal{S}_{1}+2g(0)\mathcal{S}_{5}\right)\|U\|_{\mathcal{H}}\|F\|_{\mathcal{H}}.
\end{eqnarray*}
Thus, the proof is completed.
\end{proof}
\begin{lem}\label{CA11}
For $m\in (0,1)$ and under the hypothesis $(\mathrm{\mathbf{H}})$, 
 the solution \\$U=(v,z,u^{1},u^{2},u^{3},w,k)^\top\in \mathcal{D}(\mathcal{A})$ satisfies 
\begin{equation*}
  \alpha \int_{0}^{L}|v_{x}|^{2}dx\leq \mathcal{S}_{7}\|U\|_{\mathcal{H}}\|F\|_{\mathcal{H}}.
\end{equation*}
where $\mathcal{S}_{7}=\alpha(b-c)^{2} \dfrac{\epsilon_{3}^{2}}{b^{2}\gamma^{2}}\mathcal{S}_{4}$.
\end{lem}
\begin{proof}
Inserting the equation (\ref{CA3}) in (\ref{CA7}), we get
\begin{equation}
  i\lambda u^{3}+\frac{\mu}{\epsilon_{3}}u^{1}_{x}-i\lambda \frac{\gamma}{\epsilon_{3}}v_{x}
  +\frac{\gamma}{\epsilon_{3}}f^{1}_{x}+cu^{3}=f^{5}. \label{CA12}
\end{equation}
Multiplying the result by $i\lambda^{-1}\bar{v_{x}}$, we integrate over $(0,L)$ as a result we obtain
\begin{eqnarray}
  -\int_{0}^{L}u^{3}v_{x}dx+i\lambda^{-1}\frac{\mu}{ \epsilon_{3}}\int_{0}^{L}u^{1}_{x}v_{x}dx\notag\\
  +\frac{\gamma}{\epsilon_{3}}\int_{0}^{L}|v_{x}|^{2}dx
  +i\lambda^{-1}\frac{\gamma}{\epsilon_{3}}\int_{0}^{L}f_{x}^{1}v_{x}dx\notag\\
  +i\lambda c\int_{0}^{L}u^{3}v_{x}dx&=&i\lambda^{-1}\int_{0}^{L}f^{5}v_{x}dx. \label{CA13}
\end{eqnarray}
Using the compatibility condition $\xi u^{2}_{x}-u^{3}+\frac{\gamma}{\epsilon_{3}}v_{x}=0$ 
 in the first integral of  (\ref{CA13}), we write
\begin{eqnarray}
  -\xi\int_{0}^{L}u^{2}_{x}\bar{v_{x}}dx+i\lambda^{-1}\frac{\mu}{ \epsilon_{3}}\int_{0}^{L}u^{1}_{x}v_{x}dx \notag\\
  +i\lambda^{-1}\frac{\gamma}{\epsilon_{3}}\int_{0}^{L}f_{x}^{1}v_{x}dx+i\lambda^{-1} c\int_{0}^{L}u^{3}v_{x}dx&=&i\lambda^{-1}\int_{0}^{L}f^{5}v_{x}dx.
  \label{CA14}
\end{eqnarray}
We have from (\ref{CA6}), the following equation
\begin{equation*}
i\lambda u^{2}_{x}+\frac{\mu}{\xi\epsilon_{3}}u^{1}_{x}+bu^{2}_{x}=f_{x}^{4}.
\end{equation*}
Multiplying this later with $-i\xi \lambda^{-1}\bar{v_{x}}$,  we integrate over $(0,L)$ then summing with (\ref{CA6}), one has
\begin{eqnarray*}
-b\xi\int_{0}^{L}u^{2}_{x}v_{x}dx+c\int_{0}^{L}u^{3}v_{x}dx+\frac{\gamma}{\epsilon_{3}}\int_{0}^{L}f^{1}v_{x}dx
&=&\int_{0}^{L}f^{5}v_{x}dx\\
&&-\xi\int_{0}^{L}f^{4}_{x}v_{x}dx.
\end{eqnarray*}
Using the compatibility condition $\xi u^{2}_{x}=u^{3}-\frac{\gamma}{\epsilon_{3}}v_{x}$ in this later gives the following 
\begin{eqnarray}\label{CA15}
b \frac{\gamma}{\epsilon_{3}}\int_{0}^{L}|v_{x}|^{2}dx&=&(b-c)\int_{0}^{L}u^{3}v_{x}dx-\frac{\gamma}{\epsilon_{3}}\int_{0}^{L}f^{1}_{x}v_{x}dx
\notag\\
&&+\int_{0}^{L}f^{5}v_{x}dx-\xi \int_{0}^{L}f_{x}^{4}v_{x}dx.
\end{eqnarray}
On the other side, using the compatibility condition for $F\in \mathcal{H}$, $\xi f^{4}_{x}-f^{5}+\frac{\gamma}{\epsilon_{3}}f^{1}_{x}=0$  in (\ref{CA15}) we get 
\begin{equation*}
b \dfrac{\gamma}{\epsilon_{3}}\int_{0}^{L}|v_{x}|^{2}dx=(b-c)\int_{0}^{L}u^{3}v_{x}dx,
\end{equation*}
based on the Young's inequality, we arrive from the right hand side of the previous equality to the following
\begin{eqnarray*}
  |b-c|\int_{0}^{L}|u^{3}||v_{x}|dx &\leq& \frac{(b-c)^{2} r_{1}}{2}\int_{0}^{1}|u^{3}|^{2}dx+\frac{1}{2r_{1}}\int_{0}^{1}|v_{x}|^{2}dx \\
  &\leq& \frac{(b-c)^{2} r_{1}}{2} \mathcal{S}_{4}
   \|U\|_{\mathcal{H}}\|F\|_{\mathcal{H}}+\frac{1}{2r_{1}}\int_{0}^{L}|v_{x}|^{2}dx.
\end{eqnarray*}
Taking $r_{1}=\dfrac{\epsilon_{3}}{b\gamma}$, we obtain
\begin{equation}
\alpha \int_{0}^{L}|v_{x}|^{2}dx\leq \alpha(b-c)^{2} \dfrac{\epsilon_{3}^{2}}{b^{2}\gamma^{2}}\mathcal{S}_{4}\|U\|_{\mathcal{H}}\|F\|_{\mathcal{H}}.
\end{equation}
That achieve the result.
\end{proof}
\begin{lem}\label{LEMMA3}
For $m\in (0,1)$ and under the hypothesis $(\mathrm{\mathbf{H}})$, 
 the solution\\
  $U=(v,z,u^{1},u^{2},u^{3},w,k)^\top\in \mathcal{D}(\mathcal{A})$ satisfies 
\begin{equation}
\rho \int_{0}^{L}|z|^{2}dx \leq \mathcal{S}_{8} \|U\|_{\mathcal{H}}\|F\|_{\mathcal{H}},
\end{equation}
where $\mathcal{S}_{8}=\mathcal{S}_{7}+1+\dfrac{\gamma}{2} \left(\mathcal{S}_{4}+ \mathcal{S}_{7}\right)+\dfrac{\gamma}{\sqrt{\alpha}\sqrt{\epsilon_{3}}}+\dfrac{a^{2}}{2\rho}\mathcal{S}_{1}+\dfrac{1}{\sqrt{\rho}}$.
\end{lem}
\begin{proof}
We multiply the equation (\ref{CA4}) by $-\rho i\lambda^{-1}\bar{z}$ then we integrate over $(0,L)$, we obtain
\begin{eqnarray*}
  \rho \int_{0}^{L}|z|^{2}dx+i\lambda^{-1}\alpha \int_{0}^{L}v_{xx}z dx+i \lambda^{-1}\gamma\int_{0}^{L}u_{x}^{3}\bar{z}dx\\
 -i\lambda^{-1}a\int_{0}^{L}w_{x}z dx&=&-i\lambda^{-1}\rho \int_{0}^{L}f^{2}zdx,
\end{eqnarray*}
it follows
\begin{eqnarray}
\rho \int_{0}^{L}|z|^{2}dx&=&i\lambda^{-1}\alpha \int_{0}^{L}v_{x}z_{x}dx+i \lambda^{-1}\gamma\int_{0}^{L}u^{3}\bar{z_{x}}dx\notag\\
&&+i\lambda^{-1}a\int_{0}^{L}w_{x}z dx
-i\lambda^{-1}\rho \int_{0}^{L}f^{2}zdx.\label{CA1A6}
\end{eqnarray}
Inserting the equation (\ref{CA3}) in (\ref{CA1A6}), one has
\begin{eqnarray*}
  \rho \int_{0}^{L}|z|^{2}dx &=& \alpha \int_{0}^{L}|v_{x}|^{2}dx-i\lambda^{-1}\alpha \int_{0}^{L}v_{x}f_{x}^{1}dx+\gamma\int_{0}^{L}u^{3}v_{x}dx\\
  &&-i\gamma \lambda^{-1}\int_{0}^{L}u^{3}f_{x}^{1}dx
  +ia\lambda^{-1}\int_{0}^{L}w_{x}zdx-i\lambda^{-1}\rho \int_{0}^{L}f^{2}zdx,
\end{eqnarray*}
then, we get
\begin{eqnarray}\label{CA18}
  \rho\int_{0}^{L}|z|^{2}dx&\leq& \alpha \int_{0}^{L}|v_{x}|^{2}dx+\alpha|\lambda^{-1}| \int_{0}^{1}|v_{x}||f_{x}^{1}|dx+\gamma \int_{0}^{L}|u^{3}||v_{x}|dx\notag\\
  &&+\gamma |\lambda^{-1}|\int_{0}^{L}|u^{3}||f_{x}^{1}|dx
  +a|\lambda^{-1}|\int_{0}^{L}|w_{x}||z|dx\notag\\
  &&+\rho|\lambda^{-1}|\int_{0}^{L}|f^{2}||z|dx.
\end{eqnarray}
We shall give an estimation for each term of the inequality (\ref{CA18}), hence we start by using Young's inequality, lemma \ref{CA10} and lemma \ref{CA11}, to give the following
\begin{eqnarray*}
  \gamma \int_{0}^{L}|u^{3}||v_{x}|dx&\leq& \dfrac{\gamma}{2}\int_{0}^{L}|u^{3}|^{2}dx+\dfrac{\gamma}{2}  \mathcal{S}_{7}\|U\|_{\mathcal{H}}\|F\|_{\mathcal{H}}\\
  &\leq&\dfrac{\gamma}{2} \left(\mathcal{S}_{4}+\mathcal{S}_{7}\right)\|U\|_{\mathcal{H}}\|F\|_{\mathcal{H}}
\end{eqnarray*}
and
\begin{eqnarray}
   a\int_{0}^{L}|w_{x}||z|dx&\leq&\dfrac{ar_{2}}{2}\int_{0}^{L}|w_{x}|^{2}dx+ \dfrac{a}{2r_{2}}\int_{0}^{L}|z|^{2}dx\notag\\
   &\leq&\dfrac{a r_{2}}{2}\mathcal{S}_{1}\|U\|_{\mathcal{H}}\|F\|_{\mathcal{H}}
   +\dfrac{a}{2r_{2}}\int_{0}^{L}|z|^{2}dx.\label{1}
\end{eqnarray}
On the other side, since we have $\sqrt{\alpha}\|v_{x}\|\leq \|U\|_{\mathcal{H}}$, $\sqrt{\alpha}\|\|f_{x}^{1}\|\leq \|F\|_{\mathcal{H}}$, $\sqrt{\epsilon_{3}}\|u^{3}\|\leq \|U\|_{\mathcal{H}}$ and $\sqrt{\rho}\|z\|\leq \|U\|_{\mathcal{H}}$, then one can declare that we have
\begin{eqnarray}
|\lambda^{-1}|\alpha \int_{0}^{L}|v_{x}||f_{x}^{1}|dx&\leq& \|U\|_{\mathcal{H}}\|F\|_{\mathcal{H}}, \notag\\
\gamma |\lambda^{-1}|\int_{0}^{L}|u^{3}||f_{x}^{1}|&\leq& \dfrac{\gamma}{\sqrt{\alpha}\sqrt{\epsilon_{3}}}\|U\|_{\mathcal{H}}\|F\|_{\mathcal{H}},\notag\\
 |\lambda^{-1}|\int_{0}^{L}|f^{2}||z|dx &\leq&\dfrac{1}{\rho}\|U\|_{\mathcal{H}}\|F\|_{\mathcal{H}}. \label{CAS25}
\end{eqnarray}
By choosing $r_{2}=\dfrac{a}{\rho}$, and inserting all the previous inequalities in (\ref{CA18}), we give rise to the following inequality 
\begin{eqnarray*}
\dfrac{\rho}{2} \int_{0}^{L}|z|^{2}dx&\leq&\left(\mathcal{S}_{7}+1+\dfrac{\gamma}{2} \left(\mathcal{S}_{4}
+ \mathcal{S}_{7}\right)+\dfrac{\gamma}{\sqrt{\alpha}\sqrt{\epsilon_{3}}}+\dfrac{a^{2}}{2\rho}\mathcal{S}_{1}
+\dfrac{1}{\rho}\right)
\|U\|_{\mathcal{H}}\|F\|_{\mathcal{H}}.
\end{eqnarray*}
Thus, the proof is completed.
\end{proof}
In the remaining part of this subsection, we need to estimate  the term $\mu\int_{0}^{L}|u^{1}|^{2}dx$, thereby, we give the following lemma:
\begin{lem}\label{LEMMA4}
For $m\in (0,1)$,  assume that the hypothesis $(\mathrm{\mathbf{H}})$ holds, we have the following estimation
\begin{equation*}
  \mu \int_{0}^{L}|u^{1}|^{2}dx\leq \mathcal{S}_{9}\|U\|_{\mathcal{H}}\|F\|_{\mathcal{H}},
\end{equation*}
where $\mathcal{S}_{9}=2\xi \epsilon_{3} \left( \mathcal{S}_{3}+ \dfrac{b^{2}\xi\epsilon_{3}}{2} \mathcal{S}_{3}+2\dfrac{1}{\sqrt{\xi \epsilon \mu}}+ \left(\left(\dfrac{1}{\xi}+\dfrac{\gamma }{2\xi \epsilon_{3}}\right)\mathcal{S}_{4}+\dfrac{\gamma}{2\alpha \xi \epsilon_{3}}\mathcal{S}_{7}\right)\right).$
\end{lem}
\begin{proof}
Multiplying (\ref{CA6}) by $\bar{u_{1}}$, then integrating over $(0,L)$, we get
\begin{equation*}
  i\lambda \int_{0}^{L}u^{2}\bar{u^{1}}dx+\dfrac{\mu}{\xi \epsilon_{3}}\int_{0}^{L}|u^{1}|^{2}dx+b \int_{0}^{L}u^{2}\bar{u^{1}}dx=\int_{0}^{L}f^{4}\bar{u^{1}}dx.
\end{equation*}
Multiplying (\ref{CA5}) by $\bar{u^{2}}$ integrating over $(0,L)$, we get
\begin{equation*}
  i\lambda \int_{0}^{L}u^{1}\bar{u^{2}}dx-\int_{0}^{L}|u^{2}|^{2}dx-\int_{0}^{L}u^{3}\bar{u^{2}_{x}}dx=\int_{0}^{L}f^{3}\bar{u^{2}}dx.
\end{equation*}
Summing the two previous results and taking the real part of the two identities, we obtain
\begin{eqnarray}\label{CA19}
\dfrac{\mu}{\xi \epsilon_{3}}\int_{0}^{L}|u^{1}|^{2}dx&=&\int_{0}^{L}|u^{2}|^{2}dx-\Re{\left(b\int_{0}^{L}u^{2}\bar{u^{1}}dx\right)}+
\Re{\left(\int_{0}^{L}u^{3}\bar{u_{x}^{2}}dx\right)} \notag\\
&&+\Re{\left(\int_{0}^{L}f^{4}\bar{u^{1}}dx\right)}+\Re{\left(\int_{0}^{L}f^{3}\bar{u^{2}}dx\right)}.
\end{eqnarray}
We want to estimate each part of the inequality (\ref{CA19}), for that we make use of Young's inequality and lemma \ref{CA10} to write
\begin{eqnarray*}
  \left|\Re{\left(b\int_{0}^{L}u^{2}\bar{u^{1}}dx\right)}\right| &\leq& \dfrac{b r_{3}}{2}\int_{0}^{L}|u^{2}|^{2} dx+\dfrac{b}{2r_{3}}\int_{0}^{L}|u^{1}|^{2} dx,\notag\\
  &\leq& \dfrac{b r_{3}}{2}\mathcal{S}_{3}\|U\|_{\mathcal{H}}\|F\|_{\mathcal{H}}+\dfrac{b}{2r_{3}}\int_{0}^{L}|u^{1}|^{2} dx.
\end{eqnarray*}
Since we have $\sqrt{\xi \epsilon_{3}}\|f^{4}\| \leq \|F\|_{\mathcal{H}}$, $\sqrt{\mu}\|f^{3}\| \leq \|F\|_{\mathcal{H}}$, $\sqrt{\mu}\|u^{1}\| \leq \|U\|_{\mathcal{H}}$ and $\sqrt{\xi \epsilon_{3}}\|u^{2}\| \leq \|U\|_{\mathcal{H}}$, then we can give the following integrals
\begin{eqnarray*}
  \Re{\left(\int_{0}^{L}f^{3}\bar{u^{2}}dx\right)}&\leq& \dfrac{1}{\sqrt{\xi \epsilon \mu}}
  \|U\|_{\mathcal{H}}\|F\|_{\mathcal{H}},\\
\Re{\left(\int_{0}^{L}f^{4}\bar{u^{1}}dx\right)}&\leq& \dfrac{1}{\sqrt{\xi \epsilon \mu}}
\|U\|_{\mathcal{H}}\|F\|_{\mathcal{H}}.
\end{eqnarray*}
Using the compatibility condition $\xi u^{2}_{x}-u^{3}+\dfrac{\gamma}{\epsilon_{3}}v_{x}=0$ and Young's inequality, we can write
\begin{eqnarray*}
\Re{\left(\int_{0}^{L}u^{3}\bar{u_{x}^{2}}dx\right)}&=&\dfrac{1}{\xi}\int_{0}^{L}|u^{3}|^{2}dx-\dfrac{\gamma}{\xi \epsilon_{3}}\int_{0}^{L}u^{3}v_{x}dx\notag\\
&\leq& \dfrac{1}{\xi}\mathcal{S}_{4}\|U\|_{\mathcal{H}}\|F\|_{\mathcal{H}}
+\dfrac{\gamma }{2\xi \epsilon_{3}}\int_{0}^{L}|u^{3}|^{2}dx
+\dfrac{\gamma}{2\xi \epsilon_{3}}\int_{0}^{L}|v_{x}|^{2}dx\notag\\
&\leq&\left(\left(\dfrac{1}{\xi}+\dfrac{\gamma }{2\xi \epsilon_{3}}\right)
\mathcal{S}_{4}+\dfrac{\gamma}{2\alpha \xi \epsilon_{3}}\mathcal{S}_{7}\right)
\|U\|_{\mathcal{H}}\|F\|_{\mathcal{H}}.
\end{eqnarray*}
Using the previous estimations of the right hand side parts of (\ref{CA19}), in (\ref{CA19}), then  choosing $r_{3}=\dfrac{\mu}{b\xi \epsilon_{3}}$, we get
\begin{small}
\begin{eqnarray*}
\dfrac{\mu}{\xi \epsilon_{3}}\int_{0}^{L}|u^{1}|^{2}dx&\leq&2
 \left( \mathcal{S}_{3}+ \dfrac{b^{2}\xi\epsilon_{3}}{2} \mathcal{S}_{3}+2\dfrac{1}{\sqrt{\xi \epsilon \mu}}\right.\\
 &&\left.\left(\dfrac{1}{\xi}+\dfrac{\gamma }{2\xi \epsilon_{3}}\right)\mathcal{S}_{4}
+\dfrac{\gamma}{2\alpha \xi \epsilon_{3}}\mathcal{S}_{7}\right)\|U\|_{\mathcal{H}}\|F\|_{\mathcal{H}}.
\end{eqnarray*}
\end{small}
Thus the result of the lemma holds true.
\end{proof}
\begin{proof}[Proof of theorem \ref{THEOREM1}]
First, we prove the first condition $(\mathrm{\mathbf{C1}})$ by using a contradiction argument. Suppose that the condition $(\mathrm{\mathbf{C1}})$ doesn't hold, then there exists $l\in \mathbb{R}$ such that $il\notin\rho(\mathcal{A})$, it follows that there exists a sequence $\{(\lambda_{n},U_{n})_{n\geq1}\}\subset \mathbb{R}\times\mathcal{D}(\mathcal{A})$ such that $|\lambda_{n}|\rightarrow l$ as $n\rightarrow \infty $ with $\|U_{n}\|_{\mathcal{H}}=1$ and 
\begin{equation*}
i\lambda_{n}U_{n}-\mathcal{A}U_{n}=F_{n}\rightarrow 0, \;\; in \;\; \mathcal{H}, \; as \;\; n \rightarrow \infty. 
\end{equation*}
Using lemma \ref{CA10}, lemma \ref{CA11}, lemma \ref{LEMMA3} and lemma \ref{LEMMA4}, we obtain $\|U_{n}\|_{\mathcal{H}}\rightarrow 0$ which contradict $\|U_{n}\|_{\mathcal{H}}=1$, thus the condition $(\mathrm{\mathbf{C1}})$ is satisfied.
Now, from lemma \ref{CA10}, lemma \ref{CA11}, lemma \ref{LEMMA3} and lemma \ref{LEMMA4} and  for every $U\in \mathcal{D}(\mathcal{A})$ we have
\begin{equation*}
\|U\|_{\mathcal{H}} \leq K \|F\|_{\mathcal{H}}\|U\|_{\mathcal{H}},
\end{equation*}
it follows
\begin{equation*}
\|U\|_{\mathcal{H}}=\|(i\lambda I-\mathcal{A})^{-1}F\|_{\mathcal{H}}\leq C\|F\|_{\mathcal{H}}, \;\; for \; every \;\; F\in \mathcal{H}.
\end{equation*}
Hence, the second condition $(\mathrm{\mathbf{C2}})$ holds true. thus we can apply Hung and Pruss theorem to conclude the result of theorem \ref{THEOREM1}.
\end{proof}
\subsection{The electrical field component in $z$-direction is damped $\mathbf{b=0}$ and $\mathbf{c\neq0}$}
We give the following result considering a piezoelectric beam with Coleman-Gurtin thermal 
law and the electrical field component in $z$-direction is damped :
\begin{thm}\label{THEOREM2}
For $m\in(0,1)$, under the hypothesis $(\mathrm{\mathbf{H}})$, the $\mathcal{C}_{0}-$semigroup of contractions $(e^{t\mathcal{A}})_{t\geq0}$ is exponentially stable, i.e. there exists $M\geq 1$ and $\epsilon >0$ such  that
\begin{equation*}
  \|e^{t\mathcal{A}}U_{0}\|_{\mathcal{H}}\leq Me^{-\epsilon t} \|U_{0}\|_{\mathcal{H}}, \;\; \mathrm{for} \;\; t \geq 0.
\end{equation*}
\end{thm}
Following Hung and Pr\"{u}ss in \cite{17,23}, we have to check  the following conditions :
\begin{equation*}
i\mathbb{R}\subset \rho(\mathcal{A}) \;\;\;\;\;\;\;\;\;\;\;\;\;\;\;\;\;\; \;\;\;\;\;\;\;\;\;\;\;\;\;\;\;\;\;\;\;\;\;\;\;\;\;\;\;\;\;\;(\mathrm{\mathbf{C1}})
\end{equation*}
and
\begin{equation*}
\sup_{\lambda \in \mathbb{R}}\|(i\lambda I-\mathcal{A})^{-1}\|_{\mathcal{L}(\mathcal{H})}=o(1). \;\;\;\;\;\;\;\;\;\;\;\;\;\;\;\;\;\; (\mathrm{\mathbf{C2}})
\end{equation*}
The proof of $(\mathrm{\mathbf{C2}})$ is presented below through the use of several lemmas. Let $(\lambda,U=(v,z,u^{1},u^{2},u^{3},w,\kappa)^\top)\in \mathbb{R}^{\ast}\times\mathcal{D}(\mathcal{A})$, with $|\lambda|\geq1$, such that
  \begin{equation}
  (i\lambda I- \mathcal{A})U=F:=(f^{1},f^{2},f^{3},f^{4},f^{5},f^{6},f^{7})^\top\in \mathcal{H},
  \end{equation}
  detailed as
  \begin{eqnarray}
    i\lambda v-z &=& f^{1}, \label{CAS3} \\
    i\lambda z-\frac{\alpha}{\rho}v_{xx}-\frac{\gamma}{\rho}u^{3}_{x}+\frac{a}{\rho}w_{x} &=& f^{2},\label{CAS4} \\
    i\lambda u^{1}-u^{2}+u^{3}_{x} &=& f^{3},\label{CAS5} \\
    i\lambda u^{2}+\frac{\mu}{\xi \epsilon_{3}}u^{1} &=& f^{4}, \label{CAS6}\\
    i\lambda u^{3}+\frac{\mu}{\epsilon_{3}}u^{1}_{x}-\frac{\gamma}{\epsilon_{3}}z_{x}+cu^{3} &=& f^{5}, \label{CAS7}\\
    i\lambda w-d \Lambda^{m}_{xx}+az_{x} &=& f^{6}, \label{CAS8}\\
    i\lambda k+k_{s}-w &=& f^{7}.\label{CAS9}
  \end{eqnarray}
We start with the following lemma:
\begin{lem}\label{CAS10}
For $m\in (0,1)$, assume that the condition  $(\mathrm{\mathbf{H}})$ holds. The solution $U\in \mathcal{D}(\mathcal{A})$ satisfy the following estimates
\begin{eqnarray*}
\int_{0}^{L}|w_{x}|^{2}dx&\leq&\mathcal{M}_{1}\|U\|_{\mathcal{H}}\|F\|_{\mathcal{H}},\\
\int_{0}^{L}|w|^{2}dx&\leq&\mathcal{M}_{2}\|U\|_{\mathcal{H}}\|F\|_{\mathcal{H}},\\
\int_{0}^{L}|u^{3}|^{2}dx&\leq&\mathcal{M}_{3}\|U\|_{\mathcal{H}}\|F\|_{\mathcal{H}},
\end{eqnarray*}
and
\begin{eqnarray*}
  \int_{0}^{+\infty} \int_{0}^{L}\sigma(s)|\kappa_{x}|^{2}dxds&\leq& \mathcal{M}_{4}\|U\|_{\mathcal{H}}\|F\|_{\mathcal{H}}, \\
  \int_{0}^{L}|\Lambda_{x}^{m}|^{2}dx &\leq& \mathcal{M}_{5}\|U\|_{\mathcal{H}}\|F\|_{\mathcal{H}},
\end{eqnarray*}
with
\begin{eqnarray*}
\mathcal{M}_{1}&=&\dfrac{1}{(1-m)d},\;\; \mathcal{M}_{2}=\dfrac{C_{P}}{(1-m)d},\;\;
 \mathcal{M}_{3}=\dfrac{1}{c\epsilon_{3}},\\
 \mathcal{M}_{4}&=&\frac{2}{mdd_{\sigma}},\;\; \mathcal{M}_{5}=\dfrac{2(1-m)}{d}+\dfrac{4g(0)}{mdd_{\sigma}}.
\end{eqnarray*}
\end{lem}
\begin{proof}
To get the estimations of this lemma, we proceed similarly as the proof of  lemma \ref{CA10}.
\end{proof}
\begin{lem}\label{LEMMA1}
For $m\in (0,1)$, suppose that the hypothesis ($\mathrm{\mathbf{H}}$) holds, we have the following estimation
\begin{eqnarray*}
 \alpha \int_{0}^{L}|v_{x}|^{2}dx &\lesssim&\mathcal{M}_{12} \left(|\lambda^{-1}|+1\right)\|U\|_{\mathcal{H}}\|F\|_{\mathcal{H}}\\
 &&+|\lambda^{-1}| \|U\|_{\mathcal{H}}^{\frac{1}{2}}\|F\|_{\mathcal{H}}^{\frac{1}{2}}
 \left(\left(|\lambda|+1\right)\|U\|_{\mathcal{H}}+\|F\|_{\mathcal{H}}\right)\\
&&+|\lambda^{-1}|\|F\|^{\frac{1}{4}}_{\mathcal{H}} \left( \mathbb{B}(U,F)^{2}\|U\|^{\frac{3}{4}}_{\mathcal{H}}
   +\mathbb{B}(U,F) \|U\|^{\frac{5}{4}}_{\mathcal{H}}\right),
 \end{eqnarray*}
where $\mathbb{B}(U,F)=\left( (|\lambda|^{\frac{1}{2}}+1)\|U\|^{\frac{1}{2}}_{\mathcal{H}}+ \|F\|^{\frac{1}{2}}_{\mathcal{H}}\right).$
\end{lem}
\begin{proof}
From (\ref{CAS8}) and (\ref{CAS3}), we have
\begin{equation}\label{CAS11}
i\lambda a v_{x}=-i\lambda w+d\Lambda_{xx}^{m}+f^{6}+af^{1}_{x}.
\end{equation}
Multiplying (\ref{CAS11}) by $-i\lambda^{-1}\bar{v_{x}}$, integrating over $(0,L)$, we get
\begin{eqnarray*}
a\int_{0}^{L}|v_{x}|^{2}dx&=&-\int_{0}^{L}wv_{x}dx-id\lambda^{-1}\int_{0}^{L}\Lambda_{xx}^{m}v_{x}dx\\
&&-i\lambda^{-1}\int_{0}^{L}f^{6}v_{x}dx-i\lambda^{-1}a\int_{0}^{L}f^{1}_{x}v_{x}dx,
\end{eqnarray*}
then,
\begin{eqnarray}
a\int_{0}^{L}|v_{x}|^{2}dx&=&-\int_{0}^{L}wv_{x}dx+id\lambda^{-1}\int_{0}^{L}\Lambda^{m}_{x}v_{xx}dx+i\lambda^{-1} d \Lambda_{x}(0)v_{x}(0)\notag\\
&&-i\lambda^{-1}\int_{0}^{L}f^{6}v_{x}dx
-i\lambda^{-1}a\int_{0}^{L}f^{1}_{x}v_{x}d x\notag\\
&\leq&\int_{0}^{L}|w||v_{x}|dx+d|\lambda^{-1}|\int_{0}^{L}|\Lambda^{m}_{x}||v_{xx}|dx+d|\lambda^{-1}||\Lambda^{m}_{x}(0)||v_{x}(0)| \notag\\
&&+|\lambda^{-1}| \int_{0}^{L}|f^{6}||v_{x}|dx
+a|\lambda^{-1}|\int_{0}^{L}|f_{x}^{1}||v_{x}|dx.\label{CAS20}
\end{eqnarray}
We make use of  Young's inequality and lemma \ref{CAS10} to write
\begin{eqnarray}
   \int_{0}^{L}|w||v_{x}|dx&\leq& \dfrac{\delta_{1}}{2}\int_{0}^{L}|w|^{2}dx+\dfrac{1}{2\delta_{1}}\int_{0}^{L}|v_{x}|^{2}dx\notag \\
  &\leq& \delta_{1}\dfrac{\mathcal{M}_{2}}{2}\|U\|_{\mathcal{H}}\|F\|_{\mathcal{H}}
  +\dfrac{1}{2\delta_{1}}\int_{0}^{L}|v_{x}|^{2}dx. \label{CAS18}
\end{eqnarray}
Using the inequalities $\sqrt{\alpha}\|f_{x}^{1}\|\leq \|F\|_{\mathcal{H}}$, $\|f^{6}\|\leq \|F\|_{\mathcal{H}}$
 and $\sqrt{\alpha} \|v_{x}\|\leq \|U\|_{\mathcal{H}}$ to give these following
\begin{equation}
 |\lambda^{-1}| \int_{0}^{L}|f^{6}||v_{x}|dx
+a|\lambda^{-1}|\int_{0}^{L}|f_{x}^{1}||v_{x}|dx\leq|\lambda^{-1}| \left(\dfrac{1}{\sqrt{\alpha}}
+\dfrac{a}{\alpha}\right)\|U\|_{\mathcal{H}}\|F\|_{\mathcal{H}}.\label{CAS19}
\end{equation}
From (\ref{CAS4}), we have
\begin{equation}\label{CAS12}
v_{xx}=\dfrac{\rho}{\alpha}\left[ i\lambda z-\dfrac{\gamma}{\rho}u^{3}_{x}+\dfrac{a}{\rho}w_{x}-f^{2}\right].
\end{equation}
Inserting (\ref{CAS5}) in (\ref{CAS12}), we obtain
\begin{equation}\label{CAS13}
  v_{xx}=\dfrac{i\lambda \rho}{\alpha}z+\dfrac{i\lambda \gamma}{\alpha}u^{1}-\dfrac{\gamma}{\alpha}u^{2}-\dfrac{\gamma}{\alpha}f^{3}+\dfrac{a}{\alpha}w_{x}-\dfrac{\rho}{\alpha}f^{2}.
\end{equation}
Using the inequalities $\sqrt{\rho}\|f^{2}\|^{2}\leq \|F\|_{\mathcal{H}}$, $\sqrt{\epsilon_{3}}\|f^{5}\|\leq \|F\|_{\mathcal{H}}$, $\sqrt{\xi\epsilon_{3}}\|u^{2}\|\leq \|U\|_{\mathcal{H}}$, $\sqrt{\mu}\|u^{1}\|\leq \|U\|_{\mathcal{H}}$ and Young's inequality $ab\leq \dfrac{a^{2}}{2}+\dfrac{b^{2}}{2}$, we estimate
\begin{eqnarray*}
\|v_{xx}\|&\leq& |\lambda|\dfrac{\sqrt{\rho}}{\alpha}\|U\|_{\mathcal{H}}
+|\lambda| \dfrac{\gamma}{\alpha \sqrt{\mu}}\|U\|_{\mathcal{H}}+ \dfrac{\gamma}{\alpha \sqrt{\epsilon_{3}\xi}}\|U\|_{\mathcal{H}}\\
&&+\dfrac{\gamma}{\alpha \sqrt{\mu}}\|F\|_{\mathcal{H}}+\dfrac{\sqrt{\rho}}{\alpha}\|F\|_{\mathcal{H}}
+\dfrac{a}{\alpha}\sqrt{\mathcal{M}_{1}}\|U\|_{\mathcal{H}}^{\frac{1}{2}}\|F\|_{\mathcal{H}}^{\frac{1}{2}}\\
&\leq&\mathcal{M}_{7} \left((|\lambda|+1)\|U\|_{\mathcal{H}}+\|U\|_{\mathcal{H}}^{\frac{1}{2}}\|F\|_{\mathcal{H}}^{\frac{1}{2}}+\|F\|_{\mathcal{H}}\right)\\
&\leq& \mathcal{M}_{8} \left((|\lambda|+1)\|U\|_{\mathcal{H}}+\|F\|_{\mathcal{H}}\right),
\end{eqnarray*}
where $\mathcal{M}_{7}=\max \left( \dfrac{\sqrt{\rho}}{\alpha}+\dfrac{\gamma}{\alpha \sqrt{\mu}},\dfrac{\gamma}{\alpha \sqrt{\xi \epsilon_{3}}}, \dfrac{a}{\alpha }\sqrt{\mathcal{M}_{1}}\right)$ and $\mathcal{M}_{8}=\dfrac{3}{2}\mathcal{M}_{7}$.
Using lemma \ref{CAS10}, it follows
\begin{equation}\label{CAS21}
  d|\lambda^{-1}|\int_{0}^{L}|\Lambda^{m}_{x}||v_{xx}|dx\leq d \mathcal{M}_{8}  \sqrt{\mathcal{M}_{5}}|\lambda^{-1}|\|U\|_{\mathcal{H}}^{\frac{1}{2}}\|F\|_{\mathcal{H}}^{\frac{1}{2}}\left( (|\lambda|+1)\|U\|_{\mathcal{H}}+\|F\|_{\mathcal{H}}\right).
\end{equation}
From (\ref{CAS11}), we give the following inequality
\begin{eqnarray}\label{CAS14}
  \|\Lambda_{xx}\| &\leq& \mathcal{M}_{9}\left( |\lambda|\|U\|_{\mathcal{H}}+\|F\|_{\mathcal{H}}\right),
\end{eqnarray}
where $\mathcal{M}_{9}=\dfrac{1}{d}\left(\dfrac{a}{\sqrt{\alpha}}+1\right)$.\\
We use Gagliardo-Nirenberg inequality, the inequality (\ref{CAS14}) and lemma \ref{CAS10},  we obtain
\begin{eqnarray*}
  |\Lambda_{x}(0)| &\leq& \mathcal{C}_{2}\|\Lambda_{xx}\|^{\frac{1}{2}}\|\Lambda_{x}\|^{\frac{1}{2}}+\mathcal{C}_{2}\|\Lambda_{x}\|\\
  &\leq& \max(C_{1},C_{2})\|\Lambda_{x}\|^{\frac{1}{2}}\left(\|\Lambda_{xx}\|^{\frac{1}{2}}+\|\Lambda_{x}\|^{\frac{1}{2}}\right)\\
  &\leq&  \max(C_{1},C_{2}) \mathcal{M}_{5}^{\frac{1}{4}}\|U\|^{\frac{1}{4}}_{\mathcal{H}}\|F\|^{\frac{1}{4}}_{\mathcal{H}}\left( \mathcal{M}_{9}^{\frac{1}{2}}\left( |\lambda|\|U\|_{\mathcal{H}}+\|F\|_{\mathcal{H}}\right)^{\frac{1}{2}}
  +\mathcal{M}_{5}^{\frac{1}{4}}\|U\|^{\frac{1}{4}}_{\mathcal{H}}\|F\|^{\frac{1}{4}}_{\mathcal{H}}\right),
\end{eqnarray*}
then using the fact that we have $\sqrt{a+b}\leq \sqrt{a}+\sqrt{b}$,  and Young inequality $ab\leq a^{2}+\dfrac{b^{2}}{4}$ we claim the following
\begin{eqnarray}\label{CAS15}
|\Lambda_{x}(0)|&\leq&\mathcal{M}_{10}\|U\|^{\frac{1}{4}}_{\mathcal{H}}\|F\|^{\frac{1}{4}}_{\mathcal{H}}\left( (|\lambda|^{\frac{1}{2}}+1)\|U\|^{\frac{1}{2}}_{\mathcal{H}}+ \|F\|^{\frac{1}{2}}_{\mathcal{H}}\right),
\end{eqnarray}
where $\mathcal{M}_{10}=\dfrac{5}{4}\max(C_{1},C_{2}) \mathcal{M}_{5}^{\frac{1}{4}}\max(\mathcal{M}_{9}^{\frac{1}{2}},\mathcal{M}_{5}^{\frac{1}{4}}).$\\
We use again Gagliardo-Nirenberg inequality, to estimate $|v_{x}(0)|$, we have
\begin{eqnarray}
|v_{x}(0)|&\leq& C_{1}\|v_{xx}\|^{\frac{1}{2}}\|v_{x}\|^{\frac{1}{2}}+C_{2}\|v_{x}\| \notag\\
&\leq&\dfrac{C_{1}}{\sqrt{\alpha}}\mathcal{M}_{8}^{\frac{1}{2}}\|U\|^{\frac{1}{2}}_{\mathcal{H}}
\left((|\lambda|^{\frac{1}{2}}+1)\|U\|_{\mathcal{H}}+\|F\|_{\mathcal{H}}\right)^{\frac{1}{2}}+\dfrac{C_{2}}{\sqrt{\alpha}}\|U\|_{\mathcal{H}}\notag\\
&\leq&\mathcal{M}_{11}
\left(\left((|\lambda|^{\frac{1}{2}}+1)\|U\|_{\mathcal{H}}^{\frac{1}{2}}
+\|F\|_{\mathcal{H}}^{\frac{1}{2}}\right)\|U\|^{\frac{1}{2}}_{\mathcal{H}}+\|U\|_{\mathcal{H}}\right),
 \label{CAS16}
\end{eqnarray}
  with $\mathcal{M}_{11}=\max(\dfrac{C_{1}}{\sqrt{\sqrt{\alpha}}}\mathcal{M}_{8}^{\frac{1}{2}},\dfrac{C_{2}}{\sqrt{\alpha}})$.\\
 Making use of (\ref{CAS15}) and (\ref{CAS16}), we estimate
 \begin{equation}\label{CAS17}
   d|\lambda^{-1}||\Lambda^{m}_{x}(0)||v_{x}(0)|\leq d|\lambda^{-1}|\mathcal{M}_{10}\mathcal{M}_{11}\|F\|^{\frac{1}{4}}_{\mathcal{H}} \left( \mathcal{B}(U,F)^{2}\|U\|^{\frac{3}{4}}_{\mathcal{H}}
   +\mathcal{B}(U,F) \|U\|^{\frac{5}{4}}_{\mathcal{H}}\right),
 \end{equation}
 where $\mathcal{B}(U,F)=\left( (|\lambda|^{\frac{1}{2}}+1)\|U\|^{\frac{1}{2}}_{\mathcal{H}}+ \|F\|^{\frac{1}{2}}_{\mathcal{H}}\right)$.\\
 Inserting (\ref{CAS18}), (\ref{CAS19}), (\ref{CAS21}) and (\ref{CAS17})  in (\ref{CAS20}), then choosing $\delta_{1}=\dfrac{1}{2(a-\alpha)}$ we obtain
 \begin{eqnarray*}
 \alpha \int_{0}^{L}|v_{x}|^{2}dx &\leq& \dfrac{1}{2(a-\alpha)}\dfrac{\mathcal{M}_{2}}{2}\|U\|_{\mathcal{H}}\|F\|_{\mathcal{H}}+|\lambda^{-1}| \left(\dfrac{1}{\sqrt{\alpha}}
+\dfrac{a}{\alpha}\right)\|U\|_{\mathcal{H}}\|F\|_{\mathcal{H}}\\
&&+d |\lambda^{-1}|\mathcal{M}_{8}  \sqrt{\mathcal{M}_{5}}\|U\|_{\mathcal{H}}^{\frac{1}{2}}\|F\|_{\mathcal{H}}^{\frac{1}{2}}\left( (|\lambda|+1)\|U\|_{\mathcal{H}}+\|F\|_{\mathcal{H}}\right)\\
&&+d|\lambda^{-1}| \mathcal{M}_{10}\mathcal{M}_{11}\|F\|^{\frac{1}{4}}_{\mathcal{H}} \left( \mathcal{B}(U,F)^{2}\|U\|^{\frac{3}{4}}_{\mathcal{H}}
   +\mathcal{B}(U,F) \|U\|^{\frac{5}{4}}_{\mathcal{H}}\right),
 \end{eqnarray*}
  by taking $\mathcal{M}_{12}=\max\left(\dfrac{1}{2(a-\alpha)}\dfrac{\mathcal{M}_{2}}{2}, \left(\dfrac{1}{\sqrt{\epsilon_{3}\alpha}}
+\dfrac{a}{\alpha}\right), d \mathcal{M}_{8}  \mathcal{M}_{5} ,d \mathcal{M}_{10}\mathcal{M}_{11}\right)$, we claim our result.
\end{proof}
\begin{lem}\label{LEMMA5}
For $m \in (0,1)$, assume that the condition $(\mathrm{\mathbf{H}})$ holds. Then
\begin{equation*}
  \xi\epsilon_{3}\int_{0}^{L}|u^{2}|^{2}dx\leq \mathcal{M}_{13}\left(\int_{0}^{L}|v_{x}|^{2}dx+\|U\|_{\mathcal{H}}\|F\|_{\mathcal{H}}\right),
\end{equation*}
with $\mathcal{M}_{13}=\max\left(\dfrac{\gamma^{2}(\epsilon_{3}+1)}{2\xi \epsilon_{3}},\dfrac{\epsilon_{3}+1}{2\xi}\mathcal{M}_{3}\right).$
\end{lem}
\begin{proof}
 Multiplying the compatibility condition $\xi u_{x}^{2}-u^{3}+\frac{\gamma}{\epsilon_{3}}v_{x}=0$ by $\epsilon_{3}u^{2}_{x}$, then we integrate over $(0,L)$, we obtain
\begin{equation*}
  \xi \epsilon_{3}\int_{0}^{L}|u^{2}_{x}|^{2}dx=\epsilon_{3} \int_{0}^{L}u^{2}_{x}u^{3}dx-\gamma \int_{0}^{L}u^{2}_{x}v_{x}dx,
\end{equation*}
then it follows
\begin{eqnarray*}
  \xi \epsilon_{3}\int_{0}^{L}|u^{2}_{x}|^{2}dx &\leq& \epsilon_{3}\int_{0}^{L}|u^{2}_{x}||u^{3}|dx+\gamma\int_{0}^{L}|u^{2}_{x}||v_{x}|dx\\
  &\leq&\dfrac{\epsilon_{3}+1}{ 2\delta_{2}} \int_{0}^{L}|u_{x}^{2}|^{2}dx+\dfrac{\gamma^{2}\delta_{2}}{2}\int_{0}^{L}|v_{x}|^{2}dx+
  \dfrac{\epsilon_{3}\delta_{2}}{2}\int_{0}^{L}\int_{0}^{L}|u^{3}|^{2}dx.
\end{eqnarray*}
Taking $\delta_{2}=\dfrac{\epsilon_{3}+1}{\xi\epsilon_{3}}$ then using lemma \ref{CAS10} and lemma \ref{LEMMA1}, we get
\begin{equation*}
   \xi \epsilon_{3}\int_{0}^{L}|u^{2}_{x}|^{2}dx\leq \dfrac{\gamma^{2}(\epsilon_{3}+1)}{2\xi \epsilon_{3}} \int_{0}^{L}|v_{x}|^{2}dx
   +\dfrac{\epsilon_{3}+1}{2\xi}\mathcal{M}_{3}\|U\|_{\mathcal{H}}\|F\|_{\mathcal{H}}.
\end{equation*}
Thus the proof is completed.
\end{proof}
\begin{lem}\label{LEMMA6}
For $m \in (0,1)$, assume that the hypothesis $(\mathrm{\mathbf{H}})$ holds. Then
\begin{equation*}
  \rho\int_{0}^{L}|z|^{2}dx\leq \alpha\int_{0}^{L}|v_{x}|^{2}dx+\mathcal{M}_{13}\|U\|_{\mathcal{H}}\|F\|_{\mathcal{H}},
\end{equation*}
with $\mathcal{M}_{13}=2\left(1+\dfrac{\gamma^{2}}{2\alpha} \mathcal{S}_{4}
  +\dfrac{\gamma}{\sqrt{\alpha}\sqrt{\epsilon_{3}}}+\dfrac{a^{2}}{2\rho}\mathcal{S}_{1}+\dfrac{1}{\sqrt{\rho}}\right).$
\end{lem}

\begin{proof}
The proof of the lemma at hand bears resemblance to the proof of lemma \ref{LEMMA3}, the only different is on the estimation of a $\alpha \|v_{x}\|^{2}$. Hence, i shall present only the distinguishing parts between the two cases, for the sake of clarity, we start by giving the following estimation as in lemma \ref{LEMMA3}
\begin{eqnarray*}
  \rho\int_{0}^{L}|z|^{2}dx&\leq& \alpha \int_{0}^{L}|v_{x}|^{2}dx+\alpha|\lambda^{-1}| \int_{0}^{1}|v_{x}||f_{x}^{1}|dx+\gamma \int_{0}^{L}|u^{3}||v_{x}|dx\\ \notag
  &&+\gamma |\lambda^{-1}|\int_{0}^{L}|u^{3}||f_{x}^{1}|dx
  +a|\lambda^{-1}|\int_{0}^{L}|w_{x}||z|dx+\rho|\lambda^{-1}|\int_{0}^{L}|f^{2}||w|dx.
\end{eqnarray*}
Estimating only the third term in the right hand side, we obtain
\begin{eqnarray*}
\gamma \int_{0}^{L}|u^{3}||v_{x}|dx &\leq&  \dfrac{\gamma\delta_{3}}{2}\int_{0}^{L}|u^{3}|^{2}dx+\dfrac{\gamma}{2\delta_{3}}\int_{0}^{L}|v_{x}|^{2}dx.
\end{eqnarray*}
Choosing $\delta_{3}=\dfrac{\gamma}{\alpha}$, we get by using the same argument as in the proof of (\ref{1}) and (\ref{CAS25}) in lemma \ref{LEMMA3}, to claim that we have 
\begin{eqnarray*}
  \rho\int_{0}^{L}|z|^{2}dx &\leq&\alpha \int_{0}^{L}|v_{x}|^{2}dx
  + 2\left(1+\dfrac{\gamma^{2}}{2\alpha} \mathcal{S}_{4}
  +\dfrac{\gamma}{\sqrt{\alpha}\sqrt{\epsilon_{3}}}\right.\\
  &&\left.+\dfrac{a^{2}}{2\rho}\mathcal{S}_{1}+\dfrac{1}{\sqrt{\rho}}\right)\|U\|_{\mathcal{H}}\|F\|_{\mathcal{H}}.
\end{eqnarray*}
Hence, the result holds.
\end{proof}
\begin{lem}\label{LEMMA7}
For $m \in (0,1)$, assume that the hypothesis $(\mathrm{\mathbf{H}})$ holds. Then, the solution $U\in \mathcal{D}(\mathcal{A})$ satisfy
\begin{eqnarray*}
  \mu\int_{0}^{L}|u^{1}|^{2}dx\leq \alpha \int_{0}^{L}|v_{x}|^{2}dx+\mathcal{M}_{13}\|U\|_{\mathcal{H}}\|F\|_{\mathcal{H}},
\end{eqnarray*}
with $\mathcal{M}_{13}=2\xi \epsilon_{3} \left( \mathcal{S}_{3}+ \dfrac{b^{2}\xi\epsilon_{3}}{2} \mathcal{S}_{3}+2\dfrac{1}{\sqrt{\xi \epsilon \mu}}+ \left(\dfrac{1}{\xi}+\dfrac{\gamma^{2} }{4\xi \epsilon_{3}}\right)\mathcal{S}_{4}\right).$
\end{lem}
\begin{proof}
For the proof of this lemma, we utilize a comparable methodology as that used in the lemma \ref{LEMMA4} to derive an estimate for $\mu \|u^{1}\|^{2}$.
\end{proof}
\begin{proof}[Proof of theorem \ref{THEOREM2}]
First, we will prove the first condition $(\mathrm{\mathbf{C1}})$ by using a contradiction argument. Suppose that the condition $(\mathrm{\mathbf{C1}})$ doesn't hold, then there exists $l\in \mathbb{R}$ such that $il\notin\rho(\mathcal{A})$, it follows that there exists a sequence $\{(\lambda_{n},U_{n})_{n\geq1}\}\subset \mathbb{R}\times\mathcal{D}(\mathcal{A})$ such that $|\lambda_{n}|<l$,  $|\lambda_{n}|\rightarrow l$ as $n\rightarrow \infty $ with $\|U_{n}\|_{\mathcal{H}}=1$ and 
\begin{equation*}
i\lambda_{n}U_{n}-\mathcal{A}U_{n}=F_{n}\rightarrow 0, \;\; in \;\; \mathcal{H}, \; as \;\; n \rightarrow \infty. 
\end{equation*}
We take $U=U_{n}$, $F=F_{n}$ and $\lambda_{n}=\lambda$. From lemma \ref{CAS10}, one has 
\begin{equation*}
\int_{0}^{L}|w_{x}|^{2}dx\rightarrow 0,\;\;\; \int_{0}^{L}|u^{3}|^{2}dx\rightarrow 0. 
\end{equation*}
Since  $|\lambda_{n}|<l$, $\|U_{n}\|_{\mathcal{H}}=1$ and $\|F_{n}\|_{\mathcal{H}}\rightarrow 0$, then according to lemma \ref{LEMMA1}, we get
\begin{equation*}
\mathbb{B}(U_{n},F_{n})\rightarrow |l|^{\frac{1}{2}}+1, \;\; as\;\;\; n\rightarrow \infty,
\end{equation*}
it follows that 
\begin{equation*}
\int_{0}^{L}|v^{n}_{x}|^{2}dx\rightarrow 0,\;\; as\;\;\; n\rightarrow \infty.
\end{equation*}
Using this later with the fact that we have $\|F_{n}\|\rightarrow 0$ in lemma \ref{LEMMA5}, lemma \ref{LEMMA6} and lemma \ref{LEMMA7}, we obtain 
\begin{equation*}
\int_{0}^{L}|u^{3}|^{2}dx\rightarrow 0,\;\;\; \int_{0}^{L}|z|^{2}dx\rightarrow 0, \;\;\; \int_{0}^{L}|u^{1}|^{2}dx\rightarrow 0.
\end{equation*}
Thus, we get $\|U_{n}\|\rightarrow 0$ as $n\rightarrow \infty$ which contradict $\|U_{n}\|=1$. 
Now, we will check the second one $(\mathrm{\mathbf{C2}})$. Using a contradiction argument, suppose that there exist a sequence $\{(\lambda_{n}, U_{n})\}_{n\geq1}\subset \mathbb{R}^{\ast}\times\mathcal{D}(\mathcal{A})$ , such that
with $ |\lambda|\geq1$ and $|\lambda| \rightarrow \infty$,  such that
\begin{equation*} 
\begin{tabular}{ll}
$\|U_{n}\|_{\mathcal{H}}= 1$, &\\
$i\lambda U_{n}-\mathcal{A}U_{n} =F_{n}:=
(f^{1}_{n},f^{2}_{n},f^{3}_{n},f^{4}_{n},f^{5}_{n},f^{6}_{n},f^{7}_{n},)^\top \rightarrow 0.$\;\;\; in\;\; $\mathcal{H}$. &
\end{tabular}
\end{equation*}
Using lemma \ref{CAS10} and taking into account that we have $|\lambda_{n}| \rightarrow \infty$, $\|U_{n}\|_{\mathcal{H}}= 1$ and $F_{n}\rightarrow 0 $ in $\mathcal{H}$, we get
\begin{equation}
\int_{0}^{L}|w_{n}|^{2}dx=o(1), \;\;\; \int_{0}^{L}|u^{3}_{n}|^{2}dx=o(1),\;\;\;  \int_{0}^{+\infty} \int_{0}^{L}\sigma(s)|\kappa_{x,n}|^{2}dxds=o(1).
\end{equation}
From lemma \ref{LEMMA1} and since  $|\lambda_{n}| \rightarrow \infty$, $\|U_{n}\|_{\mathcal{H}}= 1$ and $F_{n}\rightarrow 0 $ in $\mathcal{H}$, we obtain
\begin{equation*}
\mathbb{B}(U_{n},F_{n})= (|\lambda|^{\frac{1}{2}}+1)\|U\|^{\frac{1}{2}}_{\mathcal{H}}+ \|F\|^{\frac{1}{2}}_{\mathcal{H}}=o(|\lambda|^{\frac{1}{2}}),
\end{equation*}
hence
\begin{equation*}
  \int_{0}^{L}|v_{x,n}|^{2}=o(1).
\end{equation*}
From the earliest identity, lemma \ref{LEMMA5}, lemma \ref{LEMMA6} and \ref{LEMMA7}, we get
\begin{equation}
\int_{0}^{L}|u^{2}_{n}|^{2}dx=o(1),\;\;\; \int_{0}^{L}|z_{n}|^{2}dx=o(1), \;\;\; \int_{0}^{L}|u^{1}_{n}|^{2}dx=o(1).
\end{equation}
it follows $\|U_{n}\|_{\mathcal{H}}=o(1)$ which is a contradiction with $\|U_{n}\|_{\mathcal{H}}=1$. Hence, the condition $(\mathrm{\mathbf{C2}})$ holds.
\end{proof}
\subsection{The electrical field component in $x$-direction is damped 
$\mathbf{b\neq0} $ \textbf{and} $\mathbf{c=0}$}
We give the following result considering a piezoelectric beam with Coleman-Gurtin thermal 
law and the electrical field component in $x$-direction is damped :
\begin{thm}\label{THEOREM3}
For $m\in(0,1)$, under the hypothesis $(\mathrm{\mathbf{H}})$, the $\mathcal{C}_{0}-$semigroup of contractions $(e^{t\mathcal{A}})_{t\geq0}$ is exponentially stable, i.e. there exists $M\geq 1$ and $\epsilon >0$ such  that
\begin{equation*}
  \|e^{t\mathcal{A}}U_{0}\|_{\mathcal{H}}\leq Me^{-\epsilon t} \|U_{0}\|_{\mathcal{H}}, \;\; \mathrm{for} \;\; t \geq 0.
\end{equation*}
\end{thm}
Following Hung and Pr\"{u}ss in \cite{17,23}, we have to check if the following conditions hold:
\begin{equation*}
i\mathbb{R}\subset \rho(\mathcal{A}) \;\;\;\;\;\;\;\;\;\;\;\;\;\;\;\;\;\; \;\;\;\;\;\;\;\;\;\;\;\;\;\;\;\;\;\;\;\;\;\;\;\;\;\;\;\;\;\;(\mathrm{\mathbf{C1}})
\end{equation*}
and
\begin{equation*}
\sup_{\lambda \in \mathbb{R}}\|(i\lambda I-\mathcal{A})^{-1}\|_{\mathcal{L}(\mathcal{H})}=o(1). \;\;\;\;\;\;\;\;\;\;\;\;\;\;\;\;\;\; (\mathrm{\mathbf{C2}})
\end{equation*}
The proof of $(\mathrm{\mathbf{C2}})$ is presented below through the use of several lemmas.\\
  Let $(\lambda,U=(v,z,u^{1},u^{2},u^{3},w,\kappa)^\top)\in \mathbb{R}^{\ast}\times\mathcal{D}(\mathcal{A})$, with $|\lambda|\geq1$, such that
  \begin{equation}
  (i\lambda I- \mathcal{A})U=F:=(f^{1},f^{2},f^{3},f^{4},f^{5},f^{6},f^{7})^\top\in \mathcal{H},
  \end{equation}
  detailed as
  \begin{eqnarray}
    i\lambda v-z &=& f^{1}, \label{CASE3} \\
    i\lambda z-\frac{\alpha}{\rho}v_{xx}-\frac{\gamma}{\rho}u^{3}_{x}+\frac{a}{\rho}w_{x} &=& f^{2},\label{CASE4} \\
    i\lambda u^{1}-u^{2}+u^{3}_{x} &=& f^{3},\label{CASE5} \\
    i\lambda u^{2}+\frac{\mu}{\xi \epsilon_{3}}u^{1}+bu^{2} &=& f^{4}, \label{CASE6}\\
    i\lambda u^{3}+\frac{\mu}{\epsilon_{3}}u^{1}_{x}-\frac{\gamma}{\epsilon_{3}}z_{x}&=& f^{5}, \label{CASE7}\\
    i\lambda w-\lambda \Lambda^{m}_{xx}+az_{x} &=& f^{6}, \label{CASE8}\\
    i\lambda k+k_{s}-w &=& f^{7}.\label{CASE9}
  \end{eqnarray}
We start with the following lemma
\begin{lem}\label{CASE10}
For $m\in (0,1)$, assume that the hypothesis  $(\mathrm{\mathbf{H}})$ holds. The solution $U\in \mathcal{D}(\mathcal{A})$ satisfy the following estimates
\begin{eqnarray}
\int_{0}^{L}|w_{x}|^{2}dx&\leq&\mathcal{L}_{1}\|U\|_{\mathcal{H}}\|F\|_{\mathcal{H}},\\
\int_{0}^{L}|w|^{2}dx&\leq&\mathcal{L}_{2}\|U\|_{\mathcal{H}}\|F\|_{\mathcal{H}},\\
\int_{0}^{L}|u^{2}|^{2}dx&\leq&\mathcal{L}_{3}\|U\|_{\mathcal{H}}\|F\|_{\mathcal{H}},
\end{eqnarray}
and
\begin{eqnarray}
  \int_{0}^{+\infty} \int_{0}^{L}\sigma(s)|\kappa_{x}|^{2}dxds&\leq& \mathcal{L}_{4}\|U\|_{\mathcal{H}}\|F\|_{\mathcal{H}}, \\
  \int_{0}^{L}|\Lambda_{x}^{m}|^{2}dx &\leq& \mathcal{L}_{5}\|U\|_{\mathcal{H}}\|F\|_{\mathcal{H}},
\end{eqnarray}
with
\begin{eqnarray*}
\mathcal{L}_{1}&=&\dfrac{1}{(1-m)d},\;\; \mathcal{L}_{2}=\dfrac{C_{P}}{(1-m)d},\;\; \mathcal{L}_{3}=\dfrac{1}{b\xi\epsilon_{3}},\\
 \mathcal{L}_{4}&=&\frac{2}{mdd_{\sigma}},\;\; \mathcal{L}_{5}=\dfrac{2(1-m)}{d}+\dfrac{4g(0)}{mdd_{\sigma}}.
\end{eqnarray*}
\end{lem}
\begin{proof}
For the proof of the lemma, we use the same steps as in Lemma 2 to arrive at our result.
\end{proof}
\begin{lem}\label{CASE14}
For $m\in (0,1)$, assume that the hypothesis $(\mathcal{\mathbf{H}})$ holds, then the solution $U\in \mathcal{D}(\mathcal{A})$ satisfy
\begin{eqnarray*}
 \alpha \int_{0}^{L}|v_{x}|^{2}dx &\lesssim&\mathcal{M}_{12} \left(|\lambda^{-1}|+1\right)\|U\|_{\mathcal{H}}\|F\|_{\mathcal{H}}\\
 &&+|\lambda^{-1}| \|U\|_{\mathcal{H}}^{\frac{1}{2}}\|F\|_{\mathcal{H}}^{\frac{1}{2}}\left(\left(|\lambda|+1\right)\|U\|_{\mathcal{H}}+\|F\|_{\mathcal{H}}\right)\\
&&+|\lambda^{-1}|\|F\|^{\frac{1}{4}}_{\mathcal{H}} \left( \mathbb{B}(U,F)^{2}\|U\|^{\frac{3}{4}}_{\mathcal{H}}
   +\mathbb{B}(U,F) \|U\|^{\frac{5}{4}}_{\mathcal{H}}\right),
 \end{eqnarray*}
 with $\mathbb{B}(U,F)=(|\lambda|^{\frac{1}{2}}+1)\|U\|^{\frac{1}{2}}_{\mathcal{H}}+\|F\|^{\frac{1}{2}}_{\mathcal{H}}$.
\end{lem}
\begin{proof}
For the proof of this lemma is given similarly as in lemma \ref{LEMMA1}.
\end{proof}
\begin{lem}\label{LEMMA8}
For $m \in (0,1)$, assume that the condition $(\mathrm{\mathbf{H}})$ holds. Then the solution $U\in \mathcal{D}(\mathcal{A})$ satisfy
\begin{equation*}
  \epsilon_{3}\int_{0}^{L}|u^{3}|^{2}dx\leq \dfrac{\gamma^{2}}{\epsilon_{3}}\int_{0}^{L}|v_{x}|^{2}dx,
\end{equation*}
\end{lem}
\begin{proof}
Inserting the equation (\ref{CASE3}) in (\ref{CASE7}) to get the following equation
\begin{equation*}
i\lambda u^{3}+\dfrac{\mu}{\epsilon_{3}}u^{1}_{x}-i\lambda \dfrac{\gamma}{\epsilon_{3}}v_{x}+\dfrac{\gamma}{\epsilon_{3}}f^{1}_{x}=f^{5}.
\end{equation*}
Multiplying the previous equation by $-i\lambda^{-1}\epsilon_{3}\bar{u^{3}}$, then  we obtain
\begin{eqnarray}
\epsilon_{3}\int_{0}^{L}|u^{3}|^{2}dx-i\lambda^{-1}\mu \int_{0}^{L}u^{1}_{x}\bar{u^{3}}dx \notag\\
-\gamma \int_{0}^{L}v_{x}\bar{u^{3}}dx&=&-i\lambda^{-1}\epsilon_{3}\int_{0}^{L}f^{5}\bar{u^{3}}dx\notag
\\
&&+i\lambda^{-1}\gamma \int_{0}^{L}f^{1}_{x}\bar{u^{3}}dx. \label{CASE122}
\end{eqnarray}
We have from (\ref{CASE6}) the following equation
\begin{equation*}
i\lambda u^{2}_{x}+\dfrac{\mu}{\epsilon_{3}\xi}u^{1}_{x}+bu^{2}_{x}=f^{4}_{x},
\end{equation*}
Multiplying by $i\lambda^{-1}\xi \epsilon_{3}u^{3}$, the integrating over $(0,L)$, we obtain
\begin{eqnarray}\label{CASE16}
-\xi \epsilon_{3}\int_{0}^{L}u^{2}_{x}\bar{u^{3}}dx+i\lambda^{-1}\mu \int_{0}^{L}u^{1}_{x}\bar{u^{3}}dx\notag\\
+i\lambda^{-1} b \xi \epsilon_{3}\int_{0}^{L}u^{2}_{x}u^{3}dx&=&i\lambda^{-1}\xi \epsilon_{3}\int_{0}^{L}f^{4}_{x}\bar{u^{3}}dx.
\end{eqnarray}
Making use of  the compatibility condition $\xi u^{2}_{x}-u^{3}+\dfrac{\gamma}{\epsilon_{3}}v_{x}=0$ in the first term of (\ref{CASE16}), it follows
 \begin{eqnarray}\label{CASE17}
   -\epsilon_{3}\int_{0}^{L}|u^{3}|^{2}+\gamma \int_{0}^{L}v_{x}u^{3}dx\notag\\
   +i\lambda^{-1}\mu \int_{0}^{L}u^{1}_{x}\bar{u^{3}}dx
   +i\lambda^{-1} b \xi \epsilon_{3}\int_{0}^{L}u^{2}_{x}u^{3}dx&=&i\lambda^{-1}\xi \epsilon_{3}\int_{0}^{L}f^{4}_{x}\bar{u^{3}}dx.
 \end{eqnarray}
 Summing the results (\ref{CASE122}) and (\ref{CASE17}), one has
 \begin{equation}\label{CASE18}
 \xi \epsilon_{3} b\int_{0}^{L}u^{2}_{x}\bar{u^{3}}dx=\xi \epsilon_{3}\int_{0}^{L}f^{4}_{x}\bar{u^{3}}dx
 -\epsilon_{3}\int_{0}^{L}f^{5}\bar{u^{3}}dx+\gamma\int_{0}^{L}f^{1}_{x}u^{3}dx.
 \end{equation}
 Using the compatibility condition $\xi u^{2}_{x}-u^{3}+\dfrac{\gamma}{\epsilon_{3}}v_{x}=0$ in the first term of (\ref{CASE18}), we obtain
 \begin{equation*}
 \epsilon_{3} b\int_{0}^{L}|u^{3}|^{2}dx=\gamma b \int_{0}^{L}v_{x}u^{3}dx+\xi\epsilon_{3} \int_{0}^{L}f^{4}_{x}\bar{u^{3}}dx-\epsilon_{3}\int_{0}^{L}f^{5}\bar{u^{3}}dx+\gamma\int_{0}^{L}f^{1}_{x}u^{3}dx.
 \end{equation*}
We use the compatibility condition $\xi f^{4}_{x}-f^{5}+\dfrac{\gamma}{\epsilon_{3}}f^{1}_{x}=0$ for  $F\in \mathcal{H}$, in the previous equality to get the following result.
 \begin{equation*}
  b\epsilon_{3}\int_{0}^{L}|u^{3}|^{2}dx=\gamma b \int_{0}^{L}v_{x}u^{3}dx,
 \end{equation*}
 then,
 \begin{equation}\label{CASE22}
  \epsilon_{3} \int_{0}^{L}|u^{3}|^{2}dx \leq \gamma  \int_{0}^{L}|v_{x}||u^{3}|dx.
 \end{equation}
We estimate the term on the right hand side by using Young's inequality, we obtain
\begin{equation}\label{CASE21}
\gamma \int_{0}^{L}|v_{x}||u^{3}|dx \leq \dfrac{\gamma^{2}}{2\epsilon_{3}}\int_{0}^{L}|v_{x}|^{2}dx+\dfrac{\epsilon_{3}}{2}\int_{0}^{L}|u^{3}|^{2}dx.
\end{equation}
Inserting the inequality (\ref{CASE21}) in (\ref{CASE22}) we get our result.
\end{proof}
\begin{lem}\label{CASE15}
For $m \in (0,1)$, assume that the condition $(\mathrm{\mathbf{H}})$ holds. Then the solution $U\in \mathcal{D}(\mathcal{A})$ satisfy
\begin{eqnarray*}
  \rho\int_{0}^{L}|z|^{2}dx&\lesssim& \int_{0}^{L}|v_{x}|^{2}dx+\|U\|_{\mathcal{H}}\|F\|_{\mathcal{H}} +(|\lambda^{-1}|+1)\mathbb{A}(U,F)\|F\|_{\mathcal{H}},
\end{eqnarray*}
with $\mathbb{A}(U,F)=\|U\|_{\mathcal{H}}+\|F\|_{\mathcal{H}}$.
\end{lem}
\begin{proof}
Inserting  (\ref{CASE3}) in (\ref{CASE4}), then multiplying by $v$ and integrating by parts over $(0,L)$ to get
\begin{eqnarray*}
\rho\int_{0}^{L}|\lambda v|^{2}dx&=&\alpha \int_{0}^{L}|v_{x}|^{2}dx-\gamma \int_{0}^{L}u^{3}v_{x}dx\\
&&+a \int_{0}^{L}wv_{x}dx
-i\lambda\rho\int_{0}^{L}f^{1}vdx-\rho \int_{0}^{L}f^{2}vdx,
\end{eqnarray*}
then
\begin{eqnarray}\label{CASE26}
\rho\int_{0}^{L}|\lambda v|^{2}dx&\leq&\alpha \int_{0}^{L} |v_{x}|^{2}dx+\gamma \int_{0}^{L}|u^{3}||v_{x}|dx+a \int_{0}^{L}|w||v_{x}|dx \notag\\
&&+|\lambda|\int_{0}^{L}|f^{1}||v|dx+\rho \int_{0}^{L}|f^{2}||v|dx.
\end{eqnarray}
Make use of lemma \ref{CASE10} and  the inequality $ab\leq a^{2}+\dfrac{b^{2}}{4}$, it follows
\begin{eqnarray}\label{CASE23}
a \int_{0}^{L}|w||v_{x}|dx&\leq& \int_{0}^{L}|w|^{2}dx+\dfrac{a^{2}}{4}\int_{0}^{L}|v_{x}|^{2}dx \notag\\
&\leq& \mathcal{L}_{2} \|U\|_{\mathcal{H}}\|F\|_{\mathcal{H}}+\dfrac{a^{2}}{4}\int_{0}^{L}|v_{x}|^{2}dx.
\end{eqnarray}
Using Poincaré's inequality, we can estimate $\|\lambda v\|$ as follow
\begin{equation*}
  \|\lambda v\| \leq \dfrac{1}{\sqrt{\rho}}\|U\|_{\mathcal{H}}+\dfrac{C_{P}}{\sqrt{\alpha}}\|f_{x}^{1}\|
  \leq\max\left(\dfrac{1}{\sqrt{\rho}},\dfrac{C_{P}}{\sqrt{\alpha}}\right)\left(\|U\|_{\mathcal{H}}+
  \|F\|_{\mathcal{H}}\right).
\end{equation*}
On the other side, we have $\rho \int_{0}^{L}|f^{2}|^{2}dx \leq\|F\|^{2}_{\mathcal{H}}$ and $\alpha \int_{0}^{L}|f_{x}^{1}|^{2}dx\leq \|F\|^{2}_{\mathcal{H}}$, then we estimate as follow

\begin{eqnarray}\label{2}
  \rho \int_{0}^{L}|f^{2}||v|dx &\leq&\sqrt{\rho} \|F\|_{\mathcal{H}}\|v\| \notag\\
  &\leq& \dfrac{1}{|\lambda|\sqrt{\rho}}\max\left(\dfrac{1}{\sqrt{\rho}}, \dfrac{C_{P}}{\sqrt{\alpha}}\right)\left(\|F\|_{\mathcal{H}}\|U\|_{\mathcal{H}}+\|F\|_{\mathcal{H}}^{2}\right)
\end{eqnarray}
and

\begin{eqnarray}\label{CASE25}
  \rho |\lambda|\int_{0}^{L}|f^{1}||v|dx &\leq &\rho C_{P} \|f_{x}^{1}\|\|\lambda v\|\notag\\
 &\leq&\rho C_{P}\max\left(\dfrac{1}{\sqrt{\rho}\sqrt{\alpha}},\dfrac{C_{P}}{\alpha}\right)\left(\|U\|_{\mathcal{H}}+\|F\|_{\mathcal{H}}\right)\|F\|_{\mathcal{H}}.
\end{eqnarray}

Using Young's inequality  and lemma \ref{LEMMA8}  to estimate the following 
\begin{eqnarray*}
\gamma \int_{0}^{L}|u^{3}||v_{x}|dx &\leq&  \dfrac{\gamma}{2} \int_{0}^{L}|u^{3}|^{2}dx+\dfrac{\gamma}{2} \int_{0}^{L}|v_{x}|^{2}dx \\
&\leq& \dfrac{\gamma^{3}}{2\epsilon_{3}}\|F\|_{\mathcal{H}} \|U\|_{\mathcal{H}}+\dfrac{\gamma}{2}\int_{0}^{L}|v_{x}|^{2}dx.
\end{eqnarray*}
Then, using the previous inequality, the inequalities (\ref{CASE23}), (\ref{2}) and (\ref{CASE25}) in (\ref{CASE26}) we conclude the result.
\end{proof}
\begin{lem}\label{LEMMA9}
For $m \in (0,1)$, assume that the hypothesis $(\mathrm{\mathbf{H}})$ holds. Then
\begin{eqnarray*}
  \mu\int_{0}^{L}|u^{1}|^{2}dx\lesssim\int_{0}^{L}|v_{x}|^{2}dx+\|U\|_{\mathcal{H}}\|F\|_{\mathcal{H}}.
\end{eqnarray*}
\end{lem}
\begin{proof}
For the proof of this theorem, we proceed same as in lemma \ref{LEMMA4}.
\end{proof}
\begin{proof}[Proof of theorem \ref{THEOREM3}]
First, we will prove the first condition $(\mathrm{\mathbf{C1}})$ by using a contradiction argument. Suppose that the condition $(\mathrm{\mathbf{C1}})$ doesn't hold, then there exists $l\in \mathbb{R}$ such that $il\notin\rho(\mathcal{A})$, it follows that there exists a sequence $\{(\lambda_{n},U_{n})_{n\geq1}\}\subset \mathbb{R}\times\mathcal{D}(\mathcal{A})$ such that $|\lambda_{n}|<l$,  $|\lambda_{n}|\rightarrow l$ as $n\rightarrow \infty $ with $\|U_{n}\|_{\mathcal{H}}=1$ and 
\begin{equation*}
i\lambda_{n}U_{n}-\mathcal{A}U_{n}=F_{n}\rightarrow 0, \;\; in \;\; \mathcal{H}, \; as \;\; n \rightarrow \infty. 
\end{equation*}
We take $U=U_{n}$, $F=F_{n}$ and $\lambda_{n}=\lambda$. From lemma \ref{CASE10}, one has 
\begin{equation*}
\int_{0}^{L}|w_{x}|^{2}dx\rightarrow 0,\;\;\; \int_{0}^{L}|u^{3}|^{2}dx\rightarrow 0. 
\end{equation*}
Since  $|\lambda_{n}|<l$, $\|U_{n}\|_{\mathcal{H}}=1$ and $\|F_{n}\|_{\mathcal{H}}\rightarrow 0$, then according to lemma \ref{CASE14}, we get
\begin{equation*}
\mathbb{B}(U_{n},F_{n})\rightarrow |l|^{\frac{1}{2}}+1, \;\; as\;\;\; n\rightarrow \infty,
\end{equation*}
it follows that 
\begin{equation}\label{F3}
\int_{0}^{L}|v^{n}_{x}|^{2}dx\rightarrow 0,\;\; as\;\;\; n\rightarrow \infty.
\end{equation}
Using this later  in lemma \ref{LEMMA8},  we obtain 
\begin{equation*}
\int_{0}^{L}|u_{n}^{3}|^{2}dx\rightarrow 0,  \;\; as \;\; n\rightarrow \infty.
\end{equation*}
Now, using (\ref{F3}) and the fact that we have $|\lambda_{n}|<k$, $\|U_{n}\|_{\mathcal{H}}=1$ and $\|F_{n}\|_{\mathcal{H}}\rightarrow 0$, in lemma \ref{CASE15}, we get 
\begin{equation*}
\int_{0}^{L}|z_{n}|^{2}dx\rightarrow 0, \;\; as \;\; n\rightarrow \infty.
\end{equation*}
we finish the proof of $(\mathrm{\mathbf{C1}})$, by showing that 
\begin{equation*}
\int_{0}^{L}|u^{1}|^{2}dx\rightarrow 0, \;\; as\;\; n\rightarrow \infty.
\end{equation*}
using (\ref{F3}) and the fact that we have $\|U_{n}\|_{\mathcal{H}}=1$ and $\|F_{n}\|_{\mathcal{H}}\rightarrow 0$.
thus condition $(\mathrm{\mathbf{C1}})$ is satisfied. 
Then, it suffices to check the second one $(\mathrm{\mathbf{C2}})$. Using a contradiction argument, suppose that there exist a sequence $\{(\lambda_{n}, U_{n})\}_{n\geq1}\subset \mathbb{R}^{\ast}\times\mathcal{D}(\mathcal{A})$ that satisfy $ |\lambda_{n}|\geq1$ with $|\lambda_{n}| \rightarrow \infty$,  such that
\begin{equation*}  
\begin{tabular}{ll}
$\|U_{n}\|_{\mathcal{H}}= 1$, &\\
$i\lambda_{n} U_{n}-\mathcal{A}U_{n} =F_{n}:=
(f^{1}_{n},f^{2}_{n},f^{3}_{n},f^{4}_{n},f^{5}_{n},f^{6}_{n},f^{7}_{n},) \rightarrow 0.$ \;\;\; in \;\; $\mathcal{H}$&
\end{tabular}
\end{equation*}
Using lemma \ref{LEMMA1} and taking into account that we have $|\lambda_{n}| \rightarrow \infty$, $\|U_{n}\|_{\mathcal{H}}= 1$ and $F_{n}\rightarrow 0 $ in $\mathcal{H}$, we get
\begin{equation}
\int_{0}^{L}|w_{n}|^{2}dx=o(1), \;\;\; \int_{0}^{L}|u^{2}_{n}|^{2}dx=o(1),\;\;\;  \int_{0}^{+\infty} \int_{0}^{L}\sigma(s)|\kappa_{x,n}|^{2}dxds=o(1).
\end{equation}
From lemma \ref{LEMMA5} and since  $|\lambda_{n}| \rightarrow \infty$, $\|U_{n}\|_{\mathcal{H}}= 1$ and $F_{n}\rightarrow 0 $ in $\mathcal{H}$, we obtain
\begin{equation*}
\mathbb{B}(U_{n},F_{n})= (|\lambda|^{\frac{1}{2}}+1)\|U_{n}\|^{\frac{1}{2}}_{\mathcal{H}}+ \|F_{n}\|^{\frac{1}{2}}_{\mathcal{H}}=o(|\lambda|^{\frac{1}{2}}),
\end{equation*}
hence
\begin{equation*}
  \int_{0}^{L}|v_{x,n}|^{2}=o(1).
\end{equation*}
From the previous identity, lemma \ref{LEMMA8}, lemma \ref{CASE15} and \ref{LEMMA9}, we arrive to claim that
\begin{equation*}
\int_{0}^{L}|u^{3}_{n}|^{2}dx=o(1),\;\;\; \int_{0}^{L}|z_{n}|^{2}dx=o(1), \;\;\; \int_{0}^{L}|u^{1}_{n}|^{2}dx=o(1),
\end{equation*}
it follows $\|U_{n}\|_{\mathcal{H}}=o(1)$ which is a contradiction with $\|U_{n}\|_{\mathcal{H}}=1$. Thus, the condition $(\mathrm{\mathbf{C2}})$ holds.
\end{proof}
\subsection{The electrical field component in $x$- and $z$- direction are not damped $\mathbf{b=0} $ \textbf{and} $\mathbf{c=0}$}
We give the following result considering a piezoelectric beam with Coleman-Gurtin thermal 
law and the electrical field components are not damped:
\begin{thm}\label{THEOREM4}
For $m\in(0,1)$, under the hypothesis $(\mathrm{\mathbf{H}})$, the $\mathcal{C}_{0}-$semigroup of contractions $(e^{t\mathcal{A}})_{t\geq0}$ is exponentially stable, i.e. there exists $M\geq 1$ and $\epsilon >0$ such  that
\begin{equation*}
  \|e^{t\mathcal{A}}U_{0}\|_{\mathcal{H}}\leq Me^{-\epsilon t} \|U_{0}\|_{\mathcal{H}}, \;\; \mathrm{for} \;\; t \geq 0.
\end{equation*}
\end{thm}
Following Hung and Pr\"{u}ss in \cite{17,23}, we have to check if the following conditions hold:
\begin{equation*}
i\mathbb{R}\subset \rho(\mathcal{A}) \;\;\;\;\;\;\;\;\;\;\;\;\;\;\;\;\;\; \;\;\;\;\;\;\;\;\;\;\;\;\;\;\;\;\;\;\;\;\;\;\;\;\;\;\;\;\;\;(\mathrm{\mathbf{C1}})
\end{equation*}
and
\begin{equation*}
\sup_{\lambda \in \mathbb{R}}\|(i\lambda I-\mathcal{A})^{-1}\|_{\mathcal{L}(\mathcal{H})}=o(1). \;\;\;\;\;\;\;\;\;\;\;\;\;\;\;\;\;\; (\mathrm{\mathbf{C2}})
\end{equation*}
\begin{proof}
We start with the second condition $(\mathrm{\mathbf{C2}})$. Using a contradiction argument, suppose that there exist a sequence $\{(\lambda_{n}, U_{n})\}_{n\geq1}\subset \mathbb{R}^{\ast}\times\mathcal{D}(\mathcal{A})$ , such that
with $ |\lambda_{n}|\rightarrow \infty$,  such that
\begin{equation*} 
\begin{tabular}{ll}
$\|U_{n}\|_{\mathcal{H}}= 1$, &\\
$i\lambda U_{n}-\mathcal{A}U_{n} =F_{n}:=
(f^{1}_{n},f^{2}_{n},f^{3}_{n},f^{4}_{n},f^{5}_{n},f^{6}_{n},f^{7}_{n},)^\top \rightarrow 0.$\;\;\; in \;\; $\mathcal{H}$. &
\end{tabular}
\end{equation*}
For clarity we drop the index $n$. Since for $U\in \mathcal{D}(\mathcal{A})$ and  for $F \in \mathcal{H}$,  we have
 \begin{eqnarray*}
(1-m)d\int_{0}^{L}|w_{x}|^{2}dx \notag\\
-\frac{md}{2}\int_{0}^{L}\int_{0}^{\infty}\sigma^{\prime}(s)|\kappa_{x}|^{2}dsdx=-\Re(\mathcal{A}U,U)_{\mathcal{H}}&\leq& \|F\|_{\mathcal{H}}\|U\|_{\mathcal{H}}.
\end{eqnarray*}
Then from this later and Poincaré's inequality, we obtain
\begin{equation*}
\int_{0}^{L}|w_{x}|^{2}dx=o(1),\;\;\;  \int_{0}^{L}|w|^{2}dx=o(1),\;\;\; \int_{0}^{L}\int_{0}^{\infty}\sigma^{\prime}(s)|\kappa_{x}|^{2}dsdx=o(1). \label{CASE27}
\end{equation*}
Then, using the hypothesis $(\mathcal{\mathbf{H}})$ and the proceeding of lemma \ref{CA10} to conclude
\begin{eqnarray}
  \int_{0}^{+\infty} \int_{0}^{L}\sigma(s)|\kappa_{x}|^{2}dxds&=& o(1),\;\;\;
  \int_{0}^{L}|\Lambda_{x}^{m}|^{2}dx = o(1).\label{CASE27}
\end{eqnarray}
We proceed as in lemma \ref{LEMMA1}, lemma \ref{LEMMA8} and lemma \ref{CASE15} to get the following estimation, taking into account $|\lambda_{n}|\rightarrow \infty$ and $\mathcal{F}_{n}\rightarrow 0$ in the Hilbert space $\mathcal{H}$
\begin{equation}\label{CASE28}
\int_{0}^{L}|v_{x}|^{2}dx=o(1), \;\; \; \int_{0}^{L}|u^{3}|^{2}dx=o(1),\;\;\; \int_{0}^{L}|z|^{2}dx=o(1).
\end{equation}
Now,  we want to estimate $\int_{0}^{L}|u^{2}|^{2}dx$, we claim as in lemma \ref{LEMMA5} to  get
\begin{equation*}
   \xi \epsilon_{3}\int_{0}^{L}|u^{2}_{x}|^{2}dx\leq \dfrac{\gamma^{2}(\epsilon_{3}+1)}{\xi \epsilon_{3}} \int_{0}^{L}|v_{x}|^{2}dx
   +\dfrac{\epsilon_{3}+1}{2\xi}\mathcal{M}_{3}\|U\|_{\mathcal{H}}\|F\|_{\mathcal{H}}.
\end{equation*}
Then using Poincaré's inequality, the identities in (\ref{CASE28}) to the above inequality taking account that we have $\mathcal{F} \rightarrow 0$ in $\mathcal{H}$ to write
\begin{equation*}
\int_{0}^{L}|u^{2}|^{2}dx=o(1).
\end{equation*}
Similarly as in lemma \ref{LEMMA4}, we show that $\|u^{1}\|=o(1)$. We proceed similarly by using a contradiction argument as in before, to conclude easily the condition $(\mathrm{\mathbf{C1}})$. Hence, the proof is completed.
\end{proof}
\begin{rem}
Using the same justifications as in the preceding cases, we built up an exponential stability result for each of the cases of the piezoelectric Fourier law (that is the case when  $m = 0$).  
\end{rem}
\section{Polynomial stability result for piezoelectric beam with Gurtin-Pipkin thermal law }
In this section, we will consider the asymptotic behavior of solution for the piezoelectric beam  subjected 
to Gurtin-Pipkin thermal law $m=1$ and the electric field components are damped, using Borichev and Tomilov theorem (\cite{8}), we prove a polynomial decay  of solution of rate $t^{-1}$.  \\
First, for simplification we denote by $\mathcal{B}$ the unbounded linear operator in case of $m=1$ instead of $\mathcal{A}$. then we define in the Hilbert space $\mathcal{H}$, the operator $\mathcal{B}$ by
\begin{eqnarray*}
\mathcal{B}U=
 \left [\begin{array}{c}
  z\\
\frac{\alpha}{\rho}v_{xx}+\frac{\gamma}{\rho}u_{x}^{3}-\frac{a}{\rho}w_{x}\\
 u^{2}-u_{x}^{3}\\
   -\frac{\mu }{\xi \epsilon_{3}}u^{1}-bu^{2}\\
    -\frac{\mu}{\epsilon_{3}}u^{1}_{x}+\frac{\gamma}{\epsilon_{3}}z_{x}-cu^{3}\\
    d\int_{0}^{\infty}\sigma(s)\kappa_{xx}(s)ds-az_{x}\\
    -\kappa_{s}+w
\end{array}
\right ]
\end{eqnarray*}
in a domain $\mathcal{D}(\mathcal{B})$ defined as,
\begin{eqnarray*}
\mathcal{D}(\mathcal{B})&:=& \{ U=(v,z,u^{1},u^{2},u^{3},w,k)^\top\in \mathcal{H} \\ \notag
&& z\in H^{1}_{L}(0,L),\;\; v \in H^{2}(0,L)\cap H^{1}_{L}(0,L),\;\; u^{1},u^{2}\in H^{1}_{0}(0,L), \\ \notag
&&u^{3}\in H^{1}(0,L),\;\;w \in H^{1}_{0}(0,L),\;\;  \alpha v_{x}(L)+\gamma u^{3}(L)=0,\\ \notag
&&\int_{0}^{\infty}\sigma(s)\kappa_{x}(s)ds\in H^{1}(0,L), \; \kappa_{s}\in W, \; \kappa(.,0)=0\}
\end{eqnarray*}
Then, we give the polynomial stability result of the problem (\ref{C2}) in case $m=1$, in the following theorem
\begin{thm}\label{THEOREM5}
For $m=1$, assume that the hypothesis $(\mathrm{\mathbf{H}})$ holds, the $\mathcal{C}_{0}-$semigroup of contractions $(e^{t\mathcal{B}})_{t\geq0}$ is polynomially stable, i.e. there exists $M\geq 0$  such  that
\begin{equation*}
  \|e^{t\mathcal{B}}U_{0}\|_{\mathcal{H}}\leq \dfrac{M}{\sqrt{t}} \|U_{0}\|_{\mathcal{H}}, \;\; \mathrm{for} \;\; t \geq 0.
\end{equation*}
  \end{thm}
Following Borichev and Tomilov (\cite{8}), then we have to check  the following conditions :
\begin{equation*}
i\mathbb{R}\subset \rho(\mathcal{B}) \;\;\;\;\;\;\;\;\;\;\;\;\;\;\;\;\;\; \;\;\;\;\;\;\;\;\;\;\;\;\;\;\;\;\;\;\;\;\;\;\;\;\;\;\;\;\;\;(\mathrm{\mathbf{P1}})
\end{equation*}
and
\begin{equation*}
\limsup_{|\lambda| \rightarrow \infty}\dfrac{1}{|\lambda|^{2}}\|(i\lambda I-\mathcal{B})^{-1}\|_{\mathcal{L}(\mathcal{H})}<\infty. \;\;\;\;\;\;\;\;\;\;\;\;\;\;\;\;\;\; (\mathrm{\mathbf{P2}})
\end{equation*}
To prove $(\mathrm{\mathbf{P1}})$, we use contradiction argument. Indeed, suppose that $(\mathrm{\mathbf{P1}})$ is false.  Now, since we have proved that $0 \in \rho(\mathcal{B})$ in Proposition \ref{PROPO1},  then there exist $l\in \mathbb{R}$ such that $il\notin \rho(\mathcal{B})$, using remark A.3 in appendix A (See \cite{2}), it turns out that there exists a sequence $\{(\lambda_{n},U_{n})\}_{n\geq1}\subset \mathbb{R}\times \mathcal{D}(\mathcal{A})$ such that $|\lambda_{n}|\rightarrow l$  as $n\rightarrow 0$ and $|\lambda_{n}|<|l|$ with $\|U_{n}\|_{\mathcal{H}}=1$ and 
\begin{equation}\label{F1}
i\lambda_{n}U_{n}-\mathcal{B}U_{n}=F_{n}\rightarrow 0, \;\;\; in \;\; \mathcal{H},\;\;  as\;\;  n\rightarrow 0.
\end{equation} 
To get a contradiction with $\|U_{n}\|_{\mathcal{H}}=1$, we will show that $\|U_{n}\|_{\mathcal{H}}\rightarrow 0$. For simplification we drop the index n. \\
\textbf{Step.1:} First, taking the inner product of (\ref{F1}) with $U$ in $\mathcal{H}$, we get
\begin{eqnarray*}
b\xi \epsilon_{3}\int_{0}^{L}|u^{2}|^{2}dx
+c\epsilon_{3}\int_{0}^{L}|u^{3}|^{2}dx \notag\\
-\frac{d}{2}\int_{0}^{L}\int_{0}^{\infty}\sigma^{\prime}(s)|\kappa_{x}|^{2}dsdx=-\Re(\mathcal{B}U,U)_{\mathcal{H}}&\leq& \|F\|_{\mathcal{H}}\|U\|_{\mathcal{H}}\rightarrow 0.
\end{eqnarray*}
It follows
\begin{equation}\label{F2}
\xi \epsilon_{3}\int_{0}^{L}|u^{2}|^{2}dx \rightarrow 0,\;\;\; \epsilon_{3}\int_{0}^{L}|u^{3}|^{2}dx\rightarrow 0.
\end{equation}
and 
\begin{equation*}
\int_{0}^{L}\int_{0}^{\infty}\sigma^{\prime}(s)|\kappa_{x}|^{2}dsdx \rightarrow 0.
\end{equation*}
Using the hypothesis $(\mathrm{\mathbf{H}})$ in this later, we obtain 
\begin{equation*}
\int_{0}^{L}\int_{0}^{\infty}\sigma(s)|\kappa_{x}|^{2}dsdx \rightarrow 0.
\end{equation*}
Proceeding similarly as in lemma 4.4 see \cite{2}, we get 
\begin{equation*}
\int_{0}^{L}|w_{x}|^{2}\rightarrow 0,\;\;\;\; \int_{0}^{L}|w|^{2}\rightarrow 0.
\end{equation*}
Now, we proceed as similar as lemma \ref{CA11}, lemma \ref{LEMMA3} and  lemma \ref{LEMMA4}, we obtain 
\begin{equation*}
\alpha \int_{0}^{L}|v_{x}|^{2}dx \rightarrow 0,\;\;\; \rho \int_{0}^{L}|z|^{2}dx \rightarrow 0,\;\;\; \mu \int_{0}^{L}|u^{1}|^{2}dx \rightarrow 0.
\end{equation*}
Hence, a contradiction follows and the condition ($\mathrm{\mathbf{P1}}$) is satisfied.\\
By contradiction argument, we will prove the second condition $(\mathrm{\mathbf{P2}})$. Assume that the condition $(\mathrm{\mathbf{P2}})$ does not hold, then there exist a sequence  $(\lambda_{n},U_{n}=(v_{n},z_{n},u^{1}_{n},u^{2}_{n},u^{3}_{n},w_{n},\kappa_{n}))\in \mathbb{R}^{\ast}\times\mathcal{D}(\mathcal{B})$ such that $\|U_{n}\|=1$ and $|\lambda_{n}|\rightarrow \infty$ that satisfy the following
  \begin{equation}\label{SE15}
  \lambda^{2}_{n}(i\lambda_{n} I- \mathcal{B})U_{n}=F_{n}
  :=(f_{n}^{1},f_{n}^{2},f_{n}^{3},f_{n}^{4},f_{n}^{5},f_{n}^{6},f_{n}^{7}(\cdot,s))^{T}\rightarrow 0 \;\;\; \in \mathcal{H}.
  \end{equation}
  detailed as (For sake of clarity we drop the index $n$)
  \begin{eqnarray}
    i\lambda v-z &=& \lambda^{-2}f^{1}, \label{SEC3} \\
    i\lambda z-\frac{\alpha}{\rho}v_{xx}-\frac{\gamma}{\rho}u^{3}_{x}+\frac{a}{\rho}w_{x} &=& \lambda^{-2}f^{2},\label{SEC4} \\
    i\lambda u^{1}-u^{2}+u^{3}_{x} &=&\lambda^{-2} f^{3},\label{SEC5} \\
    i\lambda u^{2}+\frac{\mu}{\xi \epsilon_{3}}u^{1}+bu^{2} &=& \lambda^{-2}f^{4}, \label{SEC6}\\
    i\lambda u^{3}+\frac{\mu}{\epsilon_{3}}u^{1}_{x}-\frac{\gamma}{\epsilon_{3}}z_{x}+cu^{3}&=& \lambda^{-2}f^{5}, \label{SEC7}\\
    i\lambda w-\lambda \int_{0}^{\infty}\sigma(s)\kappa_{xx}(s)ds+az_{x} &=&\lambda^{-2} f^{6}, \label{SEC8}\\
    i\lambda k+k_{s}-w &=& \lambda^{-2}f^{7}.\label{SEC9}
\end{eqnarray}
 We present the proof of ($\mathrm{\mathbf{P2}}$) through the use of several lemmas given as follows \\
\begin{lem}\label{LEMMA10}
For $m=1$, assume that the hypothesis $(\mathrm{\mathbf{H}})$ holds, then the solution $U=(v,z,u^{1},u^{2},u^{3},w,\kappa)\in \mathcal{D}(\mathcal{A})$ satisfies
\begin{equation*}
\int_{0}^{L}|u^{2}|^{2}dx=o(\lambda^{-2}), \;\;\;  \int_{0}^{L}|u^{3}|^{2}dx=o(\lambda^{-2})
\end{equation*}
and
\begin{equation*}
\int_{0}^{\infty}\int_{0}^{L}\sigma^{\prime}(s)|\kappa_{x}|^{2}dxds=o(\lambda^{-2}),\;\;\; \int_{0}^{\infty}\int_{0}^{L}\sigma(s)|\kappa_{x}|^{2}dxds=o(\lambda^{-2}).
\end{equation*}
\end{lem}
\begin{proof}
Taking the inner product  of (\ref{SE15}) with $U\in \mathcal{D}(\mathcal{A})$ we obtain
\begin{eqnarray}
b\xi \epsilon_{3}\int_{0}^{L}|u^{2}|^{2}dx
+c\epsilon_{3}\int_{0}^{L}|u^{3}|^{2}dx\notag\\
-\frac{d}{2}\int_{0}^{L}\int_{0}^{\infty}\sigma^{\prime}(s)|\kappa_{x}|^{2}dsdx=-\Re(\mathcal{B}U,U)_{\mathcal{H}}&\leq& |\lambda^{-1}|\|F\|_{\mathcal{H}}\|U\|_{\mathcal{H}}.
\end{eqnarray}
Since $F\rightarrow 0$ in $\mathcal{H}$,  then we get
\begin{equation*}
\int_{0}^{L}|u^{2}|^{2}dx=o(\lambda^{-2}), \;\;\; \int_{0}^{L}|u^{3}|^{2}dx=o(\lambda^{-2}), \;\;\; \int_{0}^{L}\int_{0}^{\infty}\sigma^{\prime}(s)|\kappa_{x}|^{2}dsdx=o(\lambda^{-2}).
\end{equation*}
Using the hypothesis $(\mathcal{\mathbf{H}})$ in this later we conclude
\begin{equation*}
\int_{0}^{L}\int_{0}^{\infty}\sigma(s)|\kappa_{x}|^{2}dsdx=o(\lambda^{-2}).
\end{equation*}
Hence the result.
\end{proof}
\begin{lem}
For $m=1$, assume that the hypothesis $(\mathcal{\mathbf{H}})$ holds. Then for a solution \\ $U=(v,z,u^{1},u^{2},u^{3},w,\kappa)^\top \in \mathcal{D}(\mathcal{A})$ we have
\begin{equation*}
\int_{0}^{L}|w_{x}|^{2}=o(\lambda^{-2}),\;\;\; \int_{0}^{L}|w|^{2}=o(\lambda^{-2}).
\end{equation*}
\end{lem}
\begin{proof}
The proof of this lemma follows similar steps as given as in Lemma $4.9.$ (Akil in \cite{2})
\end{proof}
\begin{lem}\label{LEMMA12}
For $m=1$, assume that the hypothesis $(\mathcal{\mathbf{H}})$ holds. Then for \\$U=(v,z,u^{1},u^{2},u^{3},w,\kappa)^\top \in \mathcal{D}(\mathcal{A})$ we have
\begin{equation*}
\int_{0}^{L}|v_{x}|^{2}=o(\lambda^{-1}). 
\end{equation*}
\end{lem}
\begin{proof}
Inserting the equation (\ref{SEC3}) in (\ref{SEC7}), we get
\begin{equation}
  i\lambda u^{3}+\frac{\mu}{\epsilon_{3}}u^{1}_{x}-i\lambda \frac{\gamma}{\epsilon_{3}}v_{x}+\lambda^{-2}\frac{\gamma}{\epsilon_{3}}f^{1}_{x}+cu^{3}=\lambda^{-2}f^{5}. \label{CA12}
\end{equation}
Multiplying the result (\ref{CA12}) by $i\lambda^{-1}\bar{v_{x}}$, we integrate over $(0,L)$ we obtain
\begin{eqnarray}
  -\int_{0}^{L}u^{3}v_{x}dx+i\frac{\mu}{\lambda \epsilon_{3}}\int_{0}^{L}u^{1}_{x}v_{x}dx+\frac{\gamma}{\epsilon_{3}}\int_{0}^{L}|v_{x}|^{2}dx\\ \notag
  +i\lambda^{-3}\frac{\gamma}{\epsilon_{3}}\int_{0}^{L}f^{1}_{x}v_{x}dx+\frac{i}{\lambda}c\int_{0}^{L}u^{3}v_{x}dx
  &=&i\lambda^{-3}\int_{0}^{L}f^{5}v_{x}dx. \label{SEC13}
\end{eqnarray}
Using the compatibility condition $\xi u^{2}_{x}-u^{3}+\frac{\gamma}{\epsilon_{3}}v_{x}=0$  in the first integral of  (\ref{CA13}), we write
\begin{eqnarray}
  -\xi\int_{0}^{L}u^{2}_{x}\bar{v_{x}}dx+i\frac{\mu}{\lambda \epsilon_{3}}\int_{0}^{L}u^{1}_{x}v_{x}dx \notag\\
  +i\lambda^{-3}\frac{\gamma}{\epsilon_{3}}\int_{0}^{L}f^{1}v_{x}dx+\frac{i}{\lambda}c\int_{0}^{L}u^{3}v_{x}dx
  &=&i\lambda^{-3}\int_{0}^{L}f^{5}v_{x}dx.
  \label{SEC14}
\end{eqnarray}
We have from (\ref{SEC6}), the following equation
\begin{equation*}
i\lambda u^{2}_{x}+\frac{\mu}{\xi\epsilon_{3}}u^{1}_{x}+bu^{2}_{x}=\lambda^{-2}f_{x}^{4},
\end{equation*}
multiplying this later by $-i\xi \lambda^{-1}\bar{v_{x}}$, we integrate over $(0,L)$ then summing with (\ref{SEC14}),  one has
\begin{eqnarray*}
-b\xi\int_{0}^{L}u^{2}_{x}v_{x}dx+c\int_{0}^{L}u^{3}v_{x}dx+\lambda^{-2}\frac{\gamma}{\epsilon_{3}}\int_{0}^{L}f^{1}v_{x}dx
&=&\lambda^{-2}\int_{0}^{L}f^{5}v_{x}dx\\
&&-\lambda^{-2}\xi\int_{0}^{L}f^{4}_{x}v_{x}dx.
\end{eqnarray*}
Using again the compatibility condition $\xi u^{2}_{x}=u^{3}-\frac{\gamma}{\epsilon_{3}}v_{x}$, one gets
\begin{eqnarray}\label{SEC15}
b \frac{\gamma}{\epsilon_{3}}\int_{0}^{L}|v_{x}|^{2}dx&=&(b-c)\int_{0}^{L}u^{3}v_{x}dx
-\frac{\gamma}{\epsilon_{3}}\int_{0}^{L}f^{1}_{x}v_{x}dx\\
&&-\int_{0}^{L}f^{5}v_{x}dx-\xi \int_{0}^{L}f_{x}^{4}v_{x}dx.
\end{eqnarray}
We use the compatibility condition for $F\in \mathcal{H}$, 
$\xi f^{4}_{x}-f^{5}+\frac{\gamma}{\epsilon_{3}}f^{1}_{x}=0$,
 in (\ref{SEC15}), to get
\begin{equation}
b \dfrac{\gamma}{\epsilon_{3}}\int_{0}^{L}|v_{x}|^{2}dx=(b-c)\int_{0}^{L}u^{3}v_{x}dx,
\end{equation}
hence by using Cauchy-Schwartz inequality and lemma \ref{LEMMA10}, we conclude our  result.
\end{proof}
\begin{lem}\label{LEMMA11}
For $m=1$, assume that the hypothesis $(\mathcal{\mathbf{H}})$ holds. Then for $U \in \mathcal{D}(\mathcal{A})$ we have
\begin{equation*}
\int_{0}^{L}|z|^{2}=o(1), \;\;\; \int_{0}^{L}|u^{1}|^{2}dx=o(1).
\end{equation*}
\end{lem}
\begin{proof}
Multiplying (\ref{SEC4}) by $-\rho i\lambda^{-1}\bar{z}$, integrating over $(0,L)$,
\begin{eqnarray*}
  \rho \int_{0}^{L}|z|^{2}dx+i\lambda^{-1}\alpha \int_{0}^{L}v_{xx}z dx+i \lambda^{-1}\gamma\int_{0}^{L}u_{x}^{3}\bar{z}dx\\
 -i\lambda^{-1}a\int_{0}^{L}w_{x}z dx&=&-i\lambda^{-3}\rho \int_{0}^{L}f^{2}zdx,
\end{eqnarray*}
it follows
\begin{eqnarray}
\rho \int_{0}^{L}|z|^{2}dx&=&i\lambda^{-1}\alpha \int_{0}^{L}v_{x}z_{x}dx+i \lambda^{-1}\gamma\int_{0}^{L}u^{3}\bar{z_{x}}dx \notag\\
&&+i\lambda^{-1}a\int_{0}^{L}w_{x}z dx
-i\lambda^{-3}\rho \int_{0}^{L}f^{2}zdx.\label{CA16}
\end{eqnarray}
Using (\ref{SEC3}) in (\ref{CA16}), we obtain
\begin{eqnarray}
  \rho \int_{0}^{L}|z|^{2}dx &=& \alpha \int_{0}^{L}|v_{x}|^{2}dx-i\lambda^{-3}\alpha \int_{0}^{L}v_{x}f_{x}^{1}dx+\gamma\int_{0}^{L}u^{3}v_{x}dx\notag\\
  &&-i\gamma \lambda^{-3}\int_{0}^{L}u^{3}f_{x}^{1}dx
  +ia\lambda^{-1}\int_{0}^{L}w_{x}zdx\notag\\
  &&-i\lambda^{-3}\rho \int_{0}^{L}f^{2}zdx. \label{SEC10}
\end{eqnarray}
We estimate easily using Cauchy-Schwartz inequality, Lemma \ref{LEMMA10}, lemma \ref{LEMMA12} and lemma \ref{LEMMA11} as follow
\begin{eqnarray*}
\alpha|\lambda^{-3}|\left| \int_{0}^{1}v_{x}f_{x}^{1}dx\right|&=&o(\lambda^{-3}), \;\;\; \left|\gamma\int_{0}^{L}u^{3}v_{x}dx\right|=o(\lambda^{-2}),\\ \left|\gamma\int_{0}^{L}u^{3}f_{x}^{1}dx\right|&=&o(\lambda^{-3}), \;\;\; \left|a\lambda^{-1}\int_{0}^{L}w_{x}zdx\right|=o(\lambda^{-1}),\\
\left|\lambda^{-3}\rho \int_{0}^{L}f^{2}zdx\right|&=&o(\lambda^{-3}).
\end{eqnarray*}
Inserting these previous estimations in (\ref{SEC10}), we obtain the first result
\begin{equation*}
  \int_{0}^{L}|z|^{2}dx=o(1).
\end{equation*}\\
 Now, Multiplying (\ref{SEC6}) by $\bar{u_{1}}$, then integrating over $(0,L)$, we get
\begin{equation}\label{A123}
  i\lambda \int_{0}^{L}u^{2}\bar{u^{1}}dx+\dfrac{\mu}{\xi \epsilon_{3}}\int_{0}^{L}|u^{1}|^{2}dx+b \int_{0}^{L}u^{2}\bar{u^{1}}dx=\int_{0}^{L}f^{4}\bar{u^{1}}dx.
\end{equation}
Multiplying (\ref{SEC5}) by $\bar{u^{2}}$ integrating over $(0,L)$, we get
\begin{equation}\label{A1234}
  i\lambda \int_{0}^{L}u^{1}\bar{u^{2}}dx-\int_{0}^{L}|u^{2}|^{2}dx-\int_{0}^{L}u^{3}\bar{u^{2}_{x}}dx=\int_{0}^{L}f^{3}\bar{u^{2}}dx.
\end{equation}
Summing the results (\ref{A123}), (\ref{A1234}) then taking the real part of the two identities, we obtain
\begin{eqnarray}\label{SEC19}
\dfrac{\mu}{\xi \epsilon_{3}}\int_{0}^{L}|u^{1}|^{2}dx&=&\int_{0}^{L}|u^{2}|^{2}dx-\Re{\left(b\int_{0}^{L}u^{2}\bar{u^{1}}dx\right)}+
\Re{\left(\int_{0}^{L}u^{3}\bar{u_{x}^{2}}dx\right)}\\
&&+\Re{\left(\int_{0}^{L}f^{4}\bar{u^{1}}dx\right)}+\Re{\left(\int_{0}^{L}f^{3}\bar{u^{2}}dx\right)}.
\end{eqnarray}
We estimate as follow
\begin{eqnarray*}
\left|\Re{\left(b\int_{0}^{L}u^{2}\bar{u^{1}}dx\right)}\right|&=&o(\lambda^{-1}), \;\;\; \left|\Re{\left(\int_{0}^{L}u^{3}\bar{u_{x}^{2}}dx\right)}\right|=o(\lambda^{-1}),\\
\left|\Re{\left(\int_{0}^{L}f^{4}\bar{u^{1}}dx\right)}\right|&=&o(1),\;\;\; \left|\Re{\left(\int_{0}^{L}f^{3}\bar{u^{2}}dx\right)}\right|
=o(\lambda^{-1}).
\end{eqnarray*}
Inserting the previous estimations and lemma \ref{LEMMA10} in  (\ref{SEC19}), we obtain 
\begin{equation*}
  \int_{0}^{L}|u^{1}|^{2}dx=o(1).
\end{equation*}
\end{proof}
\begin{proof}[Proof of theorem \ref{THEOREM5}]
From lemma \ref{LEMMA10}, lemma \ref{LEMMA12} and lemma \ref{LEMMA11}, we obtain $\|U_{n}\|_{\mathcal{H}}=o(1)$ which is a contraduction with the hypothesis  $\|U_{n}\|_{\mathcal{H}}=1$, thus the proof is completed.
\end{proof}
\section{Conlusion}
In this paper, we have established the stabilization of a piezoelectric beam with magnetic effect with Coleman-Gurtin or Gurtin-Pipkin law and under Lorenz gauge conditions, several cases are studied, we summarize the results in the following table:
We remark that there is no need to control the electrical fields to achieve an exponential decay rate of solution when a Coleman-Gurtin thermal law is considered. A polynomial decay of rate $t^{-1}$  with Gurtin-Pipkin thermal law since the electrical field components are damped is established. The optimality of the polynomial rate of decay is still an open problem.  

\Addresses

\end{document}